\documentclass[a4paper, 11pt]{article}
\usepackage{amsmath,amssymb,esint,amscd,xspace,fancyhdr,color,authblk,srcltx,fontenc,bbm}
\usepackage{hyperref}
\usepackage{cite}

\makeatletter
\newcommand{\leqnomode}{\tagsleft@true}
\newcommand{\reqnomode}{\tagsleft@false}
\makeatother

\newtheorem{theorem}{Theorem}[section]

\newtheorem{corollary}[theorem]{Corollary}
\newtheorem{definition}[theorem]{Definition}

\newtheorem{lemma}[theorem]{Lemma}

\newtheorem{remark}[theorem]{Remark}

\newenvironment{proof}[1][Proof]{\textbf{#1.} }{\hfill\rule{0.5em}{0.5em}}
{\catcode`\@=11\global\let\AddToReset=\@addtoreset
\AddToReset{equation}{section}

\AddToReset{theorem}{section}

\title{Weighted distribution approach to gradient estimates for quasilinear elliptic double-obstacle problems in Orlicz spaces}
\author{Thanh-Nhan Nguyen\thanks{Department of Mathematics, Ho Chi Minh City University of Education, Ho Chi Minh City, Vietnam; \texttt{nhannt@hcmue.edu.vn}}, Minh-Phuong Tran\footnote{Corresponding author.} \thanks{Applied Analysis Research Group, Faculty of Mathematics and Statistics, Ton Duc Thang University, Ho Chi Minh City, Vietnam; \texttt{tranminhphuong@tdtu.edu.vn}}}
\date{\today}

\begin{document}
 
\maketitle
\begin{abstract}

We construct an efficient approach to deal with the global regularity estimates for a class of elliptic double-obstacle problems in Lorentz and Orlicz spaces. The motivation of this paper comes from the study on an abstract result in the viewpoint of the fractional maximal distributions  and this work also extends some regularity results proved in \cite{PN_dist} by using the weighted fractional maximal distributions (WFMDs). We further investigate a pointwise estimates of the gradient of weak solutions via fractional maximal operators and Riesz potential of data. Moreover, in the setting of the paper, we are led to the study of problems with nonlinearity is supposed to be partially weak BMO condition (is measurable in one fixed variable and only satisfies locally small-BMO seminorms in the remaining variables).

\medskip

\medskip

\medskip

\noindent 

\medskip

\noindent Keywords: double-obstacle problem; quasilinear elliptic; gradient estimate; weighted distribution; Orlicz spaces. 

\end{abstract}   
                  
\tableofcontents

\section{Introduction and Main results}
\label{sec:intro}
1.1. \textit{The problem statement.} The aim of this article is to study the global regularity estimates for weak solutions to quasilinear elliptic double-obstacle problem associated with the operator
\begin{align*}
\mathcal{L}u = -\mathrm{div}\mathcal{A}(x,\nabla u) \quad \text{in} \  \Omega,
\end{align*}
in the setting of both weighted Lorentz and Orlicz-Lorentz spaces, where $\Omega$ is an open bounded domain of $\mathbb{R}^n$ ($n \ge 2$) and $\mathcal{A}:\Omega \times \mathbb{R}^n \to \mathbb{R}^n$ is a Carath\'eodory function (that is continuous with respect to $\xi \in \mathbb{R}^n$ for almost every $x$ in $\Omega$ and measurable in $x \in \Omega$ for every $\xi$ in $\mathbb{R}^n$). Given $\psi_1$, $\psi_2$ are two fixed functions in the Sobolev space $W^{1,p}(\Omega)$ such that $\psi_1 \le \psi_2$ almost everywhere in $\Omega$ and $\psi_1 \le 0 \le \psi_2$ on $\partial \Omega$. 
More precisely, we are interested in the double-obstacle problem for operator $\mathcal{L}$ consists of finding unknown function $u \in W^{1,p}_0(\Omega)$ satisfying $\psi_1 \le u \le \psi_2$ a.e. in $\Omega$ such that
\begin{align}\label{eq:Lu_ineq}
\mathcal{L}u \le - \mathrm{div} \mathcal{B}(x,\mathbf{F}) + g, 
\end{align}
where $\mathbf{F} \in L^p(\Omega;\mathbb{R}^n)$ and $g \in L^{\frac{p}{p-1}}(\Omega)$ for $1<p<\infty$. This problem naturally comes to the variational inequality
\begin{align}\tag{$\mathbf{P}$}\label{eq:DOP}
\int_{\Omega} \langle \mathcal{A}(x,\nabla u), \nabla (u-\varphi) \rangle dx \le  \int_{\Omega} \langle \mathcal{B}(x, \mathbf{F}), \nabla (u-\varphi) \rangle dx + \int_{\Omega} g(u-\varphi)  dx, 
\end{align}
for all $\varphi \in W^{1,p}_0(\Omega)$ and $\psi_1 \le \varphi \le \psi_2$ a.e. in $\Omega$. Such function $u$ in problem~\eqref{eq:DOP} is called a weak solution to the double-obstacle problem~\eqref{eq:Lu_ineq}. Here, we assume further that $\mathcal{A}(x,\cdot)$ is differentiable for almost every $x$ in $\Omega$,  and satisfies the growth conditions: there is $0<L<\infty$ such that \begin{align}\label{eq:A1-DOP}
& \hspace{2cm} \left| \mathcal{A}(x,\zeta) \right| + \left|\langle \partial_{\zeta} \mathcal{A}(x,\zeta), \zeta \rangle \right|  \le L |\zeta|^{p-1}, \\
\label{eq:A2-DOP}
& \langle \mathcal{A}(x,\zeta_1)-\mathcal{A}(x,\zeta_2), \zeta_1 - \zeta_2 \rangle \ge L^{-1} \left(|\zeta_1| + |\zeta_2| \right)^{p-2} |\zeta_1 - \zeta_2|^2,
\end{align}
for almost every $x$ in $\Omega$ and every $\zeta$, $\zeta_1$, $\zeta_2 \in \mathbb{R}^n \setminus \{0\}$.  As usual, we notice here that $\langle \cdot,\cdot \rangle$ is understood as the standard inner product in $\mathbb{R}^n$, and $\partial_\zeta$ denotes the partial derivative with respect to $\zeta$. Further, the operator $\mathcal{B}$ is also the Carath{\'e}dory vector valued mapping satisfying
\begin{align}\label{eq:B-DOP}
& \left| \mathcal{B}(x,\zeta) \right| \le L |\zeta|^{p-1}, \quad (x,\zeta) \in \Omega \times \mathbb{R}^n.
\end{align}

In the view of calculus of variations, a solution $u$ to this problem \eqref{eq:DOP} is also closely related to the minimizer of an energy functional satisfying $\psi_1 \le u \le \psi_2$. And the appearance of such double obstacle problems is indispensable for describing many physical phenomena, such as elasticity (to find the equilibrium position of an elastic membrane with additional constraints, see \cite{RO2017}), the Stefan's problem (to describe the temperature distribution in a homogeneous medium, see \cite{Duvaut1973, Lions1976}), financial mathematics (models for pricing American options, the \emph{exercise region} or price changes for market fluctuations, see \cite{LS2009}), Tug-of-War games (to obtain an approximation to $p$-Laplacian, see \cite{CLM2017}), etc. We also refer to \cite{Friedman,KS1980,Rodfrigues1987,Troianiello} for physical motivation and mathematical methods for obstacle problems and their applications.

The main features of this paper are the assumptions on boundary of domain $\Omega$ and the nonlinearity of coefficients $\mathcal{A}$. More specifically, in order to obtain the global regularity results, $\Omega$ here is assumed to be a \emph{Reifenberg flat dommain}. As far as we know, in the geometrical sense, the boundaries of Reifenberg flat domains are locally well-approximated by planes or hyperplanes at every scale. The concept of Reifenberg flat domain is a ``minimal regularity hypothesis" assumed on the boundary $\partial\Omega$ to guarantee the main results of the geometric analysis continue to be valid in $\Omega$. Global Calder\'on-Zygmund/regularity/gradient estimates for nonlinear elliptic and parabolic equations in such flat domains were first investigated by Byun and Wang in~\cite{BW2,SSB4} and later by others in an extensive list of references (see Section \ref{sec:pre} for detailed definition and description of Reifenberg flat domain). On the other-hand, instead of the assumption of locally small BMO semi-norm in $x$ for the nonlinearity $\mathcal{A}$, in this paper we confine such small BMO condition in $(n-1)$ spatial variables of vector $x \in \mathbb{R}^n$, meanwhile no assumption on the remaining one. Particularly, the coefficients $\mathcal{A}(x,\cdot)$ is allowed to be measurable in one single variable, say $x_1$, and only satisfies locally small BMO semi-norm in the remaining variables, we say $x^*=(x_2,x_3, ..., x_n) \in \mathbb{R}^{n-1}$ (note that these spatial variables may be rearranged which shows $x=(x_1,x^*)$ in vector space $\mathbb{R}^n$). It can be seen that this additional hypothesis on $\mathcal{A}$ (that is merely measurable in one spatial variable, but regular in the others), called \emph{partially weak BMO condition}, is weaker than the small BMO condition in the whole space $\mathbb{R}^n$ considered in previous studies \cite{SSB2, BW1_1, BW2, SSB4, MP12, PNnonuniform,MPTNsub} and so on. The study of Calder\'on-Zygmund theory for linear elliptic equations with partially BMO coefficients was introduced in \cite{DK2010} and \cite{BW3}, independently. Later, results can be extended to higher order elliptic and parabolic systems by Dong and D. Kim in \cite{DK2011} and for nonlinear elliptic equations of $p$-Laplacian type by Y. Kim in \cite{Kim2018}. It is worth noticing that Y. Kim in his paper showed that this condition is the minimal regularity requirement on $\mathcal{A}$ for Calder\'on-Zygmund type estimates. In general, to establish the Calder\'on-Zygmund estimates for nonlinear elliptic/parabolic equations, under \emph{partially weak BMO condition}, the number of spatial variables in which the nonlinearity $\mathcal{A}$ assumed to be measurable cannot be larger than one. The definition of partially BMO coefficients will be described in detail in Section \ref{sec:pre} below. 

1.2. \textit{Relation to prior works.}  
Before stating the main results in this article, let us briefly review some existing contributions related to regularity estimates developed in recent years. Associated with nonlinear elliptic equations, going back to the fundamental result due to Iwaniec in \cite{Iwaniec83}, the very first nonlinear Calder\'on-Zygmund type estimates related to the elliptic $p$-Laplace equation were presented. Then, classical results of Iwaniec were extended to the case of elliptic systems of $p$-Laplacian type by DiBenedetto and Manfredi. There have been further interior and global regularity results established by several authors with suitable form of data (divergence or non-divergence form, measure data) in some certain spaces, such as \cite{Mi3, BW1, 55QH4, CM2014, CoMi2016, MPTNsub, PNJDE, MPT2018}, et cetera. Equation for the constrained problem yields the variational inequality \eqref{eq:DOP} is quasilinear elliptic equation, in which data mixed between divergence and non-divergence forms
\begin{align*}
-\mathrm{div}\mathcal{A}(x,\nabla u) &= \ -\mathrm{div}\mathcal{B}(x,\mathbf{F})+g \quad \text{in} \ \ \Omega.
\end{align*}

The global Cader\'on-Zygmund estimate has recently proved by Lee and Ok in \cite{Lee2019} motivated by preceding works by V. B\"ogelein \textit{et al.} in~\cite{BDM2011} for the case of parabolic systems of $p$-Laplacian type. Also, one of our recent advances is the extension of results in \cite{Lee2019} to the framework of Lorentz spaces in \cite{PNmix}, which was also devoted to nonlinear problems with mixed data. Over the last years, a number of intensive studies have been developed through the works of many authors, that stitched together to form a panorama of regularity theory for nonlinear elliptic/parabolic equations. For one sided obstacle problems, regularity estimates have been extensively studied over the recent decades by many authors: Choe and Lewis in  \cite{CL1991} proved $C^{0,\alpha}$ and $C^{1,\alpha}$ regularity for elliptic problems, Eleuteri in \cite{Eleuteri2007, EHL2013} considered H\"older continuity for solutions of minimization problems under standard and non-standard growth, as far as Calder\'on-Zygmund estimates for elliptic/parabolic problems (see for instance \cite{Choe2016, BDM2011, BCW2012}) and a large number of works conducted, such as \cite{Scheven2, BS2012, Baroni2014} as well as many references therein.

Reaching far beyond the literature only deals with one-sided obstacle, we send the reader to some recent advances concerning the double-obstacle problems. Let us refer to one of the very first studies, \cite{MMV1989}, in which the authors studied pointwise regularity properties of solutions to \eqref{eq:DOP} in linear case when $p=2$ and right-hand side zero. Later, a great deal of progress has been made to extend to nonlinear operators, for instance, \cite{MZ1991, KZ1991, Lieberman1991} with degenerate elliptic operators, $C^{0,\alpha}$ and $C^{1,\alpha}$ in \cite{BFM2001, Choe1992}. Recently, Calder\'on-Zygmund and regularity results for a broader class of nonlinear elliptic double-obstacle problem in certain spaces presented in \cite{RT2011, BLO2020, BR2020} and our earlier works \cite{PN_dist, PN-Schro}.

1.3. \textit{Technical tools.} 
Let us summarize here some important techniques regarding Calder\'on-Zygmund type and regularity estimates for nonlinear elliptic and parabolic partial differential equations, which have been proposed and considered by many authors during the last years. In 1983, Iwaniec in his famous work \cite{Iwaniec83} first proved the local regularity results by the use of beautiful interplay between tools from Harmonic Analysis and Nonlinear PDEs. Later Caffarelli and Peral found a different approach to the $W^{1,p}$ estimates based on Hardy-Littlewood maximal operators together with a new and refined version of Calder\'on-Zygmund lemma, presented  in \cite{CP1998}. This effective method has been widely used and developed through a vast array of contributions since then. We pay particular attention to a very successful method by Acerbi and Mingione in \cite{AM2007}, that allows to achieve a Calder\'on-Zygmund estimates in which maximal operators and harmonic analysis play no role in their proofs. And later in recent decades, the idea of this technique becomes enormously popular and has subsequently been developed in a rich literature and references at the same topic, such as \cite{55QH4, 55QH2, MP12, BW2, MPT2018, PNCCM, PNJDE, Duzamin2,KM2012, KM2014,Mi3} an so on.  One can also find an extensive list of references in the recent survey paper \cite{Mi2019}.

Motivated by such effective approach, in our previous work \cite{PN_dist}, we study a new point of view and a new approach that is more interesting for pursuing regularity theory due to the so-called \emph{fractional maximal distribution functions} (FMD). Our approach is inspired on the one hand from the essence behind the proofs of Calder\'on-Zygmund-type estimates in \cite{AM2007, Mi3, Mi1} and on the other hand by the advantages of regularity estimates in terms of fractional maximal operators, proposed in preceding papers \cite{PNJDE, PNnonuniform, PNmix}. By introducing FMD and some interesting properties on its own, we also prove the applicability of such abstract results to gradient estimates of weak solutions for both quasilinear elliptic equations and (double) obstacle problems in the same paper. 

Continuing and extending the theoretical ideas in \cite{PN_dist}, our goal in this paper is to present a weighted approach in dealing with regularity issues for elliptic double obstacle problems. By deeply using some technical tools such as the boundedness property of fractional maximal functions, reverse H\"older's inequality and basic result referred to Vitali's covering lemma (a version of Calder\'on-Zygmund decomposition), we are able to prove the level-set inequalities by specifying via \emph{weighted fractional maximal distribution functions} (WFMDs). The understanding of technical ingredients will lead us to establish a more general form of weighted regularity estimates in Lorentz and generalized Orlicz spaces, respectively. Making good use of the WFMDs, we believe that our theoretical results in this paper can provide a more complete picture in regularity for nonlinear double obstacle problems, in which some appropriate applications (that appear in many different contexts) could be explored.

1.4. \textit{Main results.} Before stating the main results in the present paper, let us introduce some important terminologies and conventions. Under some suitable assumptions on the domain $\Omega$, the leading nonlinearity is in the class of BMO functions satisfying \emph{partially weak BMO condition} (see Section \ref{sec:pre} for detailed definition and explanation), we consider the weak solution $u \in W_0^{1,p}(\Omega)$ of the variational inequality (double obstacle problem) \eqref{eq:DOP} satisfying double constraints $\psi_1 \le u \le \psi_2$ with $\psi_1$, $\psi_2 \in W^{1,p}(\Omega)$. Here, we note that the given data $\mathbf{F} \in L^p(\Omega;\mathbb{R}^n)$ and $g \in L^{\frac{p}{p-1}}(\Omega)$ for $1<p<\infty$. For the sake of simplicity, in the sequel we will often denote
\begin{align}\label{def:F}
\mathbb{F} = \left(|\nabla \psi_1|^p + |\nabla \psi_2|^p + |\mathbf{F}|^{p} + |g|^{\frac{p}{p-1}}  \right)^{\frac{1}{p}}.
\end{align}
Assuming that the nonlinear operators $\mathcal{A}$, $\mathcal{B}$ satisfy conditions~\eqref{eq:A1-DOP}-\eqref{eq:A2-DOP} and~\eqref{eq:B-DOP}. The two-obstacle problem~\eqref{eq:DOP} will be investigated in the setting of weighted spaces associated to a Muckenhoupt weight $\omega \in \mathbf{A}_{\infty}$ with notation $[\omega]_{\mathbf{A}_{\infty}} = (\nu,c_0)$. Furthermore, for brevity, we shall denote
\begin{align*}
\texttt{data} \equiv  \texttt{data}(n,p,L,[\omega]_{\mathbf{A}_{\infty}},\mathrm{diam}(\Omega)/r_0),
\end{align*}
for the dependence on a set of parameters. It is worthwhile to note here that as our main theorems below will show, the universal constant $C$ may depend on $\texttt{data}$, though it is not specified explicitly in the statements. On the other hand, throughout this paper, for a suitable regularity parameter $\delta>0$ and positive constant $r>0$, we will simply write $(\mathbb{H})^{r,\delta}$ to say that $\Omega$ is $(r,\delta)$-Reifenberg flat domain and the operator $\mathcal{A}$ satisfies the weak $(r,\delta)$ BMO condition, $[\mathcal{A}]^{1,r} \le \delta$ at the same time (see Definitions \ref{def:Reifenberg} and \ref{def:BMOcond} in Section \ref{sec:pre} below).

We are now in the position to state our main results. Firstly in Theorem \ref{theo-A}, we highlight a novelty of level-set inequality regarding to the weighted fractional maximal distribution functions in this study. Based on the WFMD inequality in Theorem \ref{theo-A}, it enable us to conclude the global regularity results in the classical Lorentz spaces via Theorem \ref{theo-B} and in Orlicz-Lorentz spaces via Theorem \ref{theo:L-O}, respectively. Here, it is worth emphasizing that in our main results, global gradient estimates are preserved under fractional maximal operators $\mathbf{M}_\alpha$ (where the `fractional derivatives' $\partial u$ of weak solutions can be controlled by the norm of the data, see \cite{KM2014}). Once our main Theorem \ref{theo-B} is stated, an obvious corollary now follows (see Corollary \ref{cor-B}). In addition, this paper also contains the pointwise estimate of weak solutions to \eqref{eq:DOP} in terms of the classical Riesz potential $\mathbf{I}_\beta$, will be also indicated in Theorem \ref{theo-B} as following. 

\begin{theorem}[Level-set inequality on WFMDs]
\label{theo-A}
For every $\alpha \in [0,n)$ and $0 < a < \frac{2}{\nu}\left(1-\frac{\alpha}{n}\right)$, one can find $\varepsilon_0 = \varepsilon_0(\alpha,a)>0$, $\delta = \delta(\alpha,a,\varepsilon)>0$ and $\sigma = \sigma(\alpha,a,\varepsilon) > 0$ such that if $(\mathcal{A},\Omega)$ satisfying assumption $(\mathbb{H})^{r_0,\delta}$, then the following weighted distribution inequality
\begin{align}\label{est:theo-A}
& \mathbf{D}^{\omega}_{\mathbf{M}_{\alpha}(|\nabla u|^p)}(\varepsilon^{-{a}}\lambda) \le C \varepsilon \mathbf{D}^{\omega}_{\mathbf{M}_{\alpha}(|\nabla u|^p)}(\lambda) +  \mathbf{D}^{\omega}_{\mathbf{M}_{\alpha}(|\mathbb{F}|^p)}(\sigma\lambda),
\end{align}
holds for every $0 < \varepsilon < \varepsilon_0$ and $\lambda>0$. Here the weighted distribution function is defined by
\begin{align}\label{eq:def-dG}
\mathbf{D}^{\omega}_{f}(\lambda) := \int_{\left\{|f|> \lambda\right\}} \omega(x) dx, \quad \mbox{ for } \lambda \ge 0.
\end{align}
\end{theorem}

\begin{theorem}[Global Lorentz estimates and pointwise regularity]
\label{theo-B}
Assume that given data $\mathbf{M}_{\alpha}(|\mathbb{F}|^p) \in L^{q,s}_{\omega}(\Omega)$ for some $0<q<\infty$, $0 < s \le \infty$ and $\alpha \in [0,n)$. Then one can find $\delta_0 = \delta_0(\alpha,q,s)>0$ such that if $(\mathcal{A},\Omega)$ satisfies assumption $(\mathbb{H})^{r_0,\delta_0}$ then $\mathbf{M}_{\alpha}(|\nabla u|^p) \in L^{q,s}_{\omega}(\Omega)$ with the following inequality
\begin{align}\label{est:theo-B}
\|\mathbf{M}_{\alpha}(|\nabla u|^p)\|_{L^{q,s}_{\omega}(\Omega)} \le C\|\mathbf{M}_{\alpha}(|\mathbb{F}|^p)\|_{L^{q,s}_{\omega}(\Omega)}.
\end{align}
Moreover, for any $\beta \in (0,n)$ and $0<t<\infty$, the following point-wise estimate
\begin{align}\label{est:theo-B-2}
\mathbf{I}_{\beta}\left(\chi_{\Omega}|\mathbf{M}_{\alpha}(|\nabla u|^p)|^t\right)(x) \le C \mathbf{I}_{\beta}\left(\chi_{\Omega}|\mathbf{M}_{\alpha}(|\mathbb{F}|^p)|^t\right)(x),
\end{align}
holds for almost everywhere $x \in \mathbb{R}^n$.
\end{theorem}

We then apply Theorem~\ref{theo-B} to the associated $\alpha = 0$ and use the boundedness property of $\mathbf{M}$ to infer the following corollary. This may be more familiar with most of the readers in the same topic.

\begin{corollary}\label{cor-B}
If given data $\mathbb{F}$ defined as in~\eqref{def:F} belongs to the weighted Lorentz space $L^{q,s}_{\omega}(\Omega)$ for some $0<q<\infty$ and $0 < s \le \infty$ then one can find $\delta_0 = \delta_0(q,s)>0$ such that if $(\mathcal{A},\Omega)$ satisfies assumption $(\mathbb{H})^{r_0,\delta_0}$ then $\nabla u \in L^{q,s}_{\omega}(\Omega)$. More precisely, there holds
\begin{align}\label{est:cor-B}
\|\nabla u\|_{L^{q,s}_{\omega}(\Omega)} \le C\|\mathbb{F}\|_{L^{q,s}_{\omega}(\Omega)}.
\end{align}
In addition, for every $\beta \in (0,n)$ the following point-wise estimate
\begin{align}\label{eq:cor-B}
\mathbf{I}_{\beta}(\chi_{\Omega}|\nabla u|^q)(x) \le C \mathbf{I}_{\beta}(\chi_{\Omega}|\mathbb{F}|^q)(x), 
\end{align}
holds for almost everywhere $x \in \mathbb{R}^n$.
\end{corollary}

\begin{theorem}[Global Orlicz-Lorentz estimates]\label{theo:L-O}
Let $\Phi$ be a Young function such that $\Phi \in \Delta_2$. Assume that $\mathbf{M}_{\alpha}(|\mathbb{F}|^p)$ belongs to the weighted Orlicz-Lorentz space $L^{\Phi;q,s}_{\omega}(\Omega)$ for some $0 < q < \infty$, $0< s \le \infty$ and $\alpha \in [0,n)$. Then one can find $\delta_0 = \delta_0(\alpha,q,s)>0$ such that if $(\mathcal{A},\Omega)$ satisfies assumption $(\mathbb{H})^{r_0,\delta_0}$ then $\mathbf{M}_{\alpha} (|\nabla u|^p) \in L^{\Phi;q,s}_{\omega}(\Omega)$ according to the inequality
\begin{align}\label{eq:L-O}
\|\mathbf{M}_{\alpha}(|\nabla u|^p)\|_{L^{\Phi;q,s}_{\omega}(\Omega)} \le C \|\mathbf{M}_{\alpha}(|\mathbb{F}|^p)\|_{L^{\Phi;q,s}_{\omega}(\Omega)}.
\end{align}
\end{theorem}

1.5. \textit{Outline of the paper.} The remainder of this article will be organized as follows. In Section~\ref{sec:pre} we introduce much of general notation, basic definitions and a few preliminary results that will be needed throughout the paper. The next section~\ref{sec:ingre} focuses on some crucial ingredients of regularity theory that will be discussed in the context of our approach. Section~\ref{sec:comparison} is devoted to proving some comparison results for double obstacle problems. For most of the research in regularity, the main difficulty is to establish comparison estimates (actually the difference between gradients of our solutions and solutions of standard homogeneous equations). An important observation is that the step of proving such comparison results is one of the key ingredients of our work. Then, in Section~\ref{sec:level-set}, we state and prove some preparatory results for the proofs of main results in Section~\ref{sec:proofs} by establishing level-set inequalities that concerning the WFMDs. 

\section{Preliminaries}
\label{sec:pre}
This preparatory section is devoted to providing some notations, conventions and basic definitions that will be essential for our main proofs later on. Moreover, we also introduce basic assumptions on problem, state and prove some preliminary results in this section.

2.1. \textbf{Notation and conventions.} In the sequel, the letter $C$ will be employed to represent a generic constant, whose value is larger or equal than one, may change from line to line during chains of estimates. The dependencies of $C$ on special parameters will be suitably emphasized between parentheses. In what follows, according to the standard notation, the Lebesgue measure of a measurable set $K \subset \mathbb{R}^n$ is denoted by $|K|$ and we will use the denotation $\displaystyle{\fint_{K}{h dx}}=\displaystyle{\frac{1}{|K|}\int_{K}{h dx}}$ as the integral average of a measurable map $h \in L^1_{\mathrm{loc}}(K)$. In the paper, $\Omega$ will denote an open bounded domain in $\mathbb{R}^n$, $n \ge 2$ and an arbitrary open ball in $\mathbb{R}^n$ of center $\xi$ and radius $\rho>0$ is the set $\{z \in \mathbb{R}^n: |z-\xi|<\rho\}$, is simply abbreviated as $B_\rho(\xi)$. Further, we also set $\Omega_{\rho}(\xi) := B_{\rho}(\xi) \cap \Omega$, and when the center $\xi \in \partial\Omega$, it can be seen as the ``surface ball'' in $\mathbb{R}^n$. Throughout the paper, by an abuse of notation whenever confusion does not arise, the set $\{x \in \Omega: |h(x)|>\tau\}$ is also written as $\{|h|>\tau\}$ for short.\\[3pt]

2.2. \textbf{Assumptions on domain and coefficients.}
\begin{definition}[$(r_0,\delta)$-Reifenberg]\label{def:Reifenberg}
For $0 < \delta < 1$ and $r_0>0$, $\Omega$ is called a $(r_0,\delta)$-Reifenberg flat domain or $\Omega$ is $(r_0,\delta)$-Reifenberg for brevity if for each $\xi \in \partial \Omega$ and each $\varrho \in (0,r_0]$, it is possible to find a coordinate system $\{y_1,y_2,...,y_n\}$ with origin at $\xi$ such that
\begin{align*}
B_{\varrho}(\xi) \cap \{y_n > \delta \varrho\} \subset B_{\varrho}(\xi) \cap \Omega \subset B_{\varrho}(\xi) \cap \{y_n > -\delta \varrho\}.
\end{align*}
\end{definition}

\begin{definition}[Partially weak $(r_0,\delta)-\mathrm{BMO}$ condition]\label{def:BMOcond}
The operator $\mathcal{A}$ is called that satisfying a partially weak $(r_0,\delta)-\mathrm{BMO}$ condition with respect to $\delta >0$ and $r_0>0$ if
\begin{align}\label{cond:BMO}
[\mathcal{A}]^{1,r_0} := \sup_{y \in \mathbb{R}^n, \, \varrho \in (0, r_0]} \fint_{B_{\varrho}(y)} \theta_1\left(\mathcal{A},B_{\varrho}(y)\right)(x) dx  \le \delta.
\end{align}
Here the function $\theta_1$ defined by
\begin{align}\label{def:theta}
\theta_1\left(\mathcal{A},B_{\varrho}(y)\right)(x) = \sup_{\xi \in \mathbb{R}^n \setminus \{0\}} \frac{|\mathcal{A}(x,\xi) - \overline{\mathcal{A}}_{B^*_{\varrho}(y^*)}(x_1,\mu)|}{|\xi|^{p-1}},
\end{align}
where $x = (x_1, x^*) \in \mathbb{R}^n$ with $x^* = (x_2, x_3,..., x_n)$, and $\overline{\mathcal{A}}_{B^*_{\varrho}(y^*)}$ denotes the integral average of $\mathcal{A}$ in $B^*_{\varrho}(y^*)$, i.e.
\begin{align*}
\overline{\mathcal{A}}_{B^*_{\varrho}(y^*)} = \fint_{B^*_{\varrho}(y^*)} \mathcal{A}(x_1,x^*,\mu) dx^*.
\end{align*}
\end{definition}
\begin{remark}
\label{rem:BMO_cond}
As aforementioned in the introductory section, this type of condition is weaker than the small $(r,\delta)$-BMO condition on operator $\mathcal{A}$ (assumed on the whole space $\mathbb{R}^n$) and therefore, leading to this new assumption, results will cover a larger class of problems with coefficient operators $\mathcal{A}$ considered in \cite{BW2, BW1_1}. It means that there is no regularity requirement in one variable $x_i$, $1 \le i \le n$, with a little abuse of notation, we say $x_1$. It can be highly oscillatory (or be a big jump moving) along the $x_1$-direction and the small BMO semi-norm only assumed in $x^*=(x_2, x_3,..., x_n)$. In \cite{BR2020}, authors used a different terminology of this condition, named  $(r,\delta)$-vanishing of co-dimension one. We recommend the readers to \cite{BW3, DK2010, DK2011, Kim2018} for detailed explanations of such requirement. This kind of assumption has its own significance, for instance, to discuss mathematical representations of models of elastic laminates or composite materials, see \cite{CKV1986, LN2003} and references relating directly the topic.
\end{remark}
\begin{definition}[Muckenhoupt classes]\label{def:Muck}
A non-negative measurable function $\omega \in L^p_{\mathrm{loc}}(\mathbb{R}^n)$ is called belonging to $\mathbf{A}_p$ with $ p \in [1, \infty)$,  if $[\omega]_{\mathbf{A}_p} < \infty$, where
\begin{align*}
[\omega]_{\mathbf{A}_p} := \sup_{B_{\varrho}(\xi) \subset \mathbb{R}^n} \left(\fint_{B_{\varrho}(\xi)} \omega(z) dz\right)\left(\fint_{B_{\varrho}(\xi)}\omega(z)^{-\frac{1}{p-1}}dz\right)^{p-1},
\end{align*}
if $ p \in (1, \infty)$ and
\begin{align*}
[\omega]_{\mathbf{A}_1} := \sup_{B_{\varrho}(\xi) \subset \mathbb{R}^n} \left(\fint_{B_{\varrho}(\xi)} \omega(z) dz\right) \sup_{z \in B_{\varrho}(\xi)} [\omega(z)]^{-1}.
\end{align*}
In particular, when $p = \infty$ we say that $\omega \in \mathbf{A}_{\infty}$ if there exist constants $c_0,\nu>0$ satisfying
\begin{align*}
\omega(K) \le c_0 \left(\frac{|K|}{|B|}\right)^\nu \omega(B), 
\end{align*}
for any measurable subset $K$ of arbitrary ball $B$ in $\mathbb{R}^n$, where $\omega(K):=\int_{K}\omega(z)dz$. In this case, we write $[\omega]_{\mathbf{A}_\infty} = (c_0,\nu)$. 
\end{definition}

Such $\omega$ satisfies Definition \ref{def:Muck} is called a Muckenhoupt weight. We also remark here two  standard properties of the Muckenhoupt classes: $\mathbf{A}_1 \subset \mathbf{A}_p \subset \mathbf{A}_\infty$ for all $1 < p < \infty$ and $\mathbf{A}_\infty = \displaystyle{\bigcup_{p<\infty}\mathbf{A}_p}.$

2.3. \textbf{Other definitions and Remarks.} In this section, we also give some further definitions concerning the main results of this paper.
\begin{definition}[Weighted Lorentz spaces]
Let $0<q<\infty$, $0<s\le \infty$ and a Muckenhoupt weight $\omega \in \mathbf{A}_{\infty}$. The weighted Lorentz space $L^{q,s}_{\omega}(\Omega)$ is the set which contains all of $f \in L^1_{\mathrm{loc}}(\Omega)$ satisfying $\|f\|_{L^{q,s}_{\omega}(\Omega)}$ is finite, where
\begin{align*}
\|f\|_{L^{q,s}_{\omega}(\Omega)} := \left[ q \int_0^\infty{ \lambda^{s-1}\omega ( \{\xi \in \Omega: |f(\xi)|>\lambda\} )^{\frac{s}{q}} d\lambda} \right]^{\frac{1}{s}} , 
\end{align*}
if $s < \infty$ and
\begin{align*}
\|f\|_{L^{q,\infty}_{\omega}(\Omega)} :=  \sup_{\lambda>0}{\lambda \omega(\{\xi \in \Omega:|f(\xi)|>\lambda\})^{\frac{1}{q}}}.
\end{align*}
\end{definition}
It can be seen that when $\omega \equiv 1$, the weighted Lorentz space $L^{q,s}_{\omega}(\Omega)$ becomes the Lorentz space $L^{q,s}(\Omega)$. Moreover, in a special case, the weighted Lorentz space $L^{q,q}_{\omega}(\Omega)$ coincides to the well-known weighted Lebesgue space $L^{q}_{\omega}(\Omega)$ which contains all of measurable function $f$ satisfying
$$\|f\|_{L^{q}_{\omega}(\Omega)} := \left(\int_{\Omega} |f(z)|^q \omega(z) dz\right)^{\frac{1}{q}} < \infty.$$

\begin{definition}[Weighted Orlicz-Lorentz spaces]
Let $q \in (0,\infty)$, $0< s \le \infty$ and $\Phi \in \Delta_2$ be a Young function. A measurable functions $f$ is called belonging to the weighted Orlicz-Lorentz class $\mathcal{O}^{\Phi; q,s}_{\omega}(\Omega)$ when $\|\Phi(|f|)\|_{L^{q,s}_{\omega}(\Omega)} < \infty$. 

The weighted Orlicz-Lorentz space $L^{\Phi;q,s}_{\omega}(\Omega)$ is known as the smallest linear subspace that contains $\mathcal{O}^{\Phi;q,s}_{\omega}(\Omega)$, equipped to the Luxemburg norm
 \begin{align*}
 \|f\|_{L^{\Phi;q,s}_{\omega}(\Omega)} = \inf \left\{t: \ t >0 \ \mbox{ satisfying } \ \left\|\Phi\left({t}^{-1}{|f|}\right)\right\|_{L^{q,s}_{\omega}(\Omega)} \le 1\right\}.
 \end{align*}
\end{definition}

\begin{definition}[Maximal operators]\label{def:Malpha}
Let $0 \le \alpha \le n$, we denote by $\mathbf{M}_{\alpha}$ the fractional maximal operator of $f \in L^1_{\mathrm{loc}}(\mathbb{R}^n)$, which is given by
\begin{align} \nonumber
\mathbf{M}_\alpha f(y) = \sup_{\varrho>0}{{\varrho}^{\alpha} \fint_{B_{\varrho}(y)}{|f(z)|dz}}, \quad y \in \mathbb{R}^n.
\end{align}
Remark that $\mathbf{M}_0$ is exactly the Hardy-Littlewood operator $\mathbf{M}$ which is studied in many literature.
\end{definition}

\begin{definition}[Riesz potential]
Let $\beta \in (0,n)$ and $f \in L^1_{\mathrm{loc}}(\mathbb{R}^n;\mathbb{R}^+)$, the fractional integral operator (or Riesz potential) of $f$, denoted by $\mathbf{I}_\beta$, is given as
\begin{align}\label{def:Riesz}
\mathbf{I}_{\beta}f(x) = \int_{\mathbb{R}^n}{\frac{f(\xi)d\xi}{|x-\xi|^{n-\beta}}}, \quad x \in \mathbb{R}^n.
\end{align}
\end{definition}

\section{Main ingredients for regularity estimates}
\label{sec:ingre}
In this section, we discuss the main ingredients in our strategy to prove regularity results for double obstacle problems \eqref{eq:DOP}. Proofs are based on the following three ingredients: some level-set inequalities on weighted fractional maximal distributions (WFDMs), a type of Vitali's covering lemma, the construction of reference homogeneous problem - to a reverse H\"older's inequality. We shall describe each of these key ingredients of this approach briefly below.

3.1. \textit{Level-set inequalities on WFDMs.} One of the main ingredient used in this paper is the level-set inequality performed on the so-called (weighted) FDMs. More precisely, in this paper, we depart from the approach discussed in \cite{PN_dist} and explore more on a weighted version. 
\begin{definition}
Let $\omega \in \mathbf{A}_{\infty}$ and a given ball $B \subset \mathbb{R}^n$. For $\lambda \ge 0$ the weighted distribution function of a measurable mapping $f$ associated to $\omega$ in $B$ is defined by
\begin{align}\label{eq:def-dG}
\mathbf{D}^{\omega}_{f}(\lambda; B) := \omega\left(\left\{x \in \Omega \cap B: \ |f(x)|> \lambda\right\}\right). 
\end{align}
In particular, we will write $\mathbf{D}^{\omega}_{f}(\lambda)$ instead of $\mathbf{D}^{\omega}_{f}(\lambda; B)$ when the open ball $B$ contains $\Omega$.
\end{definition}

Here, for two given measurable functions $\mathcal{F}$ and $\mathcal{G}$, the important point is that we try to construct/prove a level-set decay estimates of the type
\begin{align}\label{eq:dist_ex}
\mathbf{D}^\omega_\mathcal{G} (\varepsilon^{-a} \lambda; B) \le C\varepsilon\mathbf{D}^\omega_\mathcal{G}(\lambda; B) + \mathbf{D}^\omega_\mathcal{F}(\sigma_{\varepsilon}\lambda; B),
\end{align}
holds for any $0<\varepsilon \ll 1$, $a \in (0,1)$, and $\sigma_\varepsilon>0$ depending only on $\varepsilon$, $a$, to conclude the gradient estimates of weak solutions, especially in terms of $\mathbf{M}_\alpha$ (as we shall see later, the level-set inequality \eqref{eq:dist_ex} involving fractional maximal operators in $\mathcal{F}$ and $\mathcal{G}$). For the sake of clarity, in section \ref{sec:level-set}, we shall exclusively concentrate our attention on the use of weighted distribution functions to prove level-set inequalities on WFMDs. This work naturally extends the recent paper \cite{PN_dist} to the double obstacle problems and weighted estimates.

3.2. \textit{Covering Lemma.} In this study, a version of Calder\'on-Zygmund (or Vitali type) covering lemma is in used: the substitution of Calder\'on-Zygmund-Krylov-Safonov decomposition, that is more convenient for us to use balls instead of cubes. This lemma is a standard argument of measure theory. 
\begin{lemma}[covering lemma]\label{lem:cover-lem}
Assume that $\Omega$ is $(r_0,\delta)$-Reifenberg and $\omega \in \mathbf{A}_{\infty}$. Suppose that two measurable subsets $\mathcal{S}\subset \mathcal{R}$ of $\Omega$ satisfying two following hypotheses:
\begin{itemize}
\item[i)] $\omega\left(\mathcal{S}\right) \le \varepsilon \omega\left(B_{r_0}\right)$ for given $\varepsilon \in (0,1)$;
\item[ii)] for any  $0 < \varrho \le r_0$ and $\xi \in \Omega$, if $\omega\left(\mathcal{S} \cap B_{\varrho}(\xi)\right) > \varepsilon \omega\left(B_{\varrho}(\xi)\right)$ then $B_{\varrho}(\xi) \cap \Omega \subset \mathcal{R}$.
\end{itemize}
Then one can find $C>0$ such that $\omega\left(\mathcal{S}\right)\leq C \varepsilon \omega\left(\mathcal{R}\right)$.
\end{lemma} 

To our knowledge, such well-known lemma of Calder\'on-Zygmund has been widely used in many works and developed through the years with several modified versions. The current version, Lemma \ref{lem:cover-lem} plays an important role in our main proofs in this paper. We refer to~\cite[Lemma 4.2]{CC1995} or \cite{CP1998,KS1980} for further reading on this lemma and its proof. 

3.3. \textit{The construction of reference homogeneous problem.} This crucial key step yields a reverse Holder's inequality that allows us to obtain local comparison estimates between weak solutions in the interior and on the boundary of domain (stated and proved in Section \ref{sec:comparison}). The main idea is that, due to a Gehring type lemma, it enables us to confirm the higher integrability for the gradient of weak solution $V$ to homogeneous equations of the type
\begin{align}
\label{eq:ref_prob}
-\mathrm{div}\mathcal{A}(x_1, x^*, \nabla V) = 0 \ \ \text{in} \ B,
\end{align}
where $B$ is any ball whose center belonging to $\bar{\Omega}$. A very interesting result proved in \cite[Theorem 2.1]{Kim2018}, stated that for any $\gamma \ge 1$, there exists a small $\delta$ depending on $n,p,q$ and the structure of $\mathcal{A}$ such that if $V$ is a unique solution to the reference problem \eqref{eq:ref_prob} and $\mathcal{A}$ satisfies the partially weak $(\rho,\delta)$-BMO condition, then
\begin{align*}
\left(\fint_{B_\rho}{|\nabla V|^{\gamma p} dx}\right)^{\frac{1}{\gamma p}} \le C\left(\fint_{B_{2\rho}}{ |\nabla V|^p}dx\right)^{\frac{1}{p}},
\end{align*}
for any $B_{2\rho} \subset \Omega$ and the constant $C$ depends only on $n,p$ and the structure of $\mathcal{A}$. The reader is referred to \cite{Kim2018,BW3} for proofs and references.

3.4. \textit{Properties of fractional maximal functions.} Properties of Hardy-Littlewood maximal function and its fractional operators play a crucial role for gradient estimates of the weak solution to our problem.  The maximal function has been successfully used in studying regularity theory of partial differential equations. In \cite{Duzamin2}, F. Duzaar and G. Mingione first presented the gradient estimates employing fractional maximal functions and nonlinear potentials. The study of regularity estimates via fractional maximal operators has already been established in our previous paper \cite{PNJDE} by using the so-called \emph{cutoff fractional maximal operators} and later in other works \cite{PNJDE, PN_dist, PNmix, PN-Schro, PNnonuniform}. An advantage of dealing with $\mathbf{M}_\alpha$ is that one can conclude both size and oscillations of our solutions, their derivatives including \emph{fractional derivatives} $\partial^\alpha u$ controlled  by given data $\mathbf{F}$, see \cite{KM2014}. Therefore, one of the main ingredients in our proofs is the boundedness property of the fractional maximal function $\mathbf{M}_\alpha$. We will use the following lemma, whose detailed proof can be found in \cite{PNnonuniform}.
\begin{lemma}\label{lem:bound-M}
For any $\alpha \in [0,n)$ and $s \ge 1$, if $f \in L^1_{\mathrm{loc}}(\mathbb{R}^n)$ and $\alpha s <  n$ then there holds
\begin{align*}
|\left\{z \in \mathbb{R}^n:  \ \mathbf{M}_{\alpha} f(z) > \lambda \right\}| \le \left(\frac{C}{\lambda^s} \int_{\mathbb{R}^n} |f(z)|^s dz\right)^{\frac{n}{n-\alpha s}}.
\end{align*}
\end{lemma}

\section{Comparison results for double-obstacle problems}\label{sec:comparison}

This section is intended to establish some comparison estimates, that make them necessary to derive the estimates for solutions to our problem \eqref{eq:DOP} later. Let us start by proving the next lemma which gives a local comparison gradient estimate between a weak solution $u$ to problem~\eqref{eq:DOP} with the unique solution $v$ solved the corresponding quasi-linear homogeneous equations. 
\begin{lemma}\label{lem:u-v-DOP}
Let us consider $u \in W^{1,p}_0(\Omega)$ as a solution to problem~\eqref{eq:DOP} and an open ball $B \subset \mathbb{R}^n$ satisfying $\Omega_{B}:= B \cap \Omega \neq \emptyset$. Assume that $v \in u + W^{1,p}_0(\Omega_{B})$ solves the following equations
\begin{align}\label{eq:v-local}
\mathcal{L}(v)  = \  0  \ \mbox{ in } \ \Omega_{B}, \ \mbox{ and } \ v  =  u \ \mbox{ on } \ \partial \Omega_{B}. 
\end{align}
Then for any $\varepsilon \in (0, 1)$ one may find a positive number $C$ which still depends on $\varepsilon$ such that
\begin{align}\label{est:u-v-DOP}
\fint_{\Omega_{B}} |\nabla v - \nabla u|^p dx \le \varepsilon  \fint_{\Omega_{B}} |\nabla u|^p dx + C \fint_{\Omega_{B}} |\mathbb{F}|^p dx.
\end{align}
\end{lemma}
\begin{proof}
The idea is to build the comparison inequalities between gradients of $u$ and several functions solved the {\it one obstacle problem} and the {\it homogeneous equations}, respectively. Our proof here will be divided into three steps. 

{\it The first step: comparison with the one obstacle problem.} Let us consider $u_1 \in u + W_0^{1,p}(\Omega_{B})$ and $u_1 \ge \psi_1$ a.e. in $\Omega_{B}$ as the unique solution to one-sided obstacle problem as follows
\begin{align*}
\mathcal{L}(u_1) \le  \mathcal{L}(\psi_2), \quad \mbox{ in } \Omega_{B}.
\end{align*}
The corresponding variational inequality of this problem is written by
\begin{align}\label{eq:DOP-1}
\int_{\Omega_{B}} \langle \mathcal{A}(x,\nabla u_1), \nabla (u_1 -  \varphi) \rangle dx \le \int_{\Omega_{B}} \langle \mathcal{A}(x,\nabla \psi_2), \nabla ( u_1 - \varphi) \rangle dx,
\end{align}
for all $\varphi \in u + W_0^{1,p}(\Omega_{B})$ and $\varphi \ge \psi_1$ a.e in $\Omega_{B}$. Note that one may take $\varphi = u_1 - (u_1-\psi_2)^{+}$ in~\eqref{eq:DOP-1} to point out that
\begin{align*}
\int_{\Omega_{B}} \left \langle \mathcal{A}(x,\nabla u_1) - \mathcal{A}(x,\nabla \psi_2), \nabla \left((u_1 - \psi_2)^{+}\right) \right \rangle dx \le 0,
\end{align*}
and make use of~\eqref{eq:A2-DOP}, it yields
\begin{align}\label{est:DOP-3}
\int_{D} (|\nabla u_1| + |\nabla \psi_2|)^{p-2}|\nabla u_1 - \nabla \psi_2|^2 dx \le 0, 
\end{align}
where $D = \{x\in \Omega_{B}: \ u_1 \ge \psi_2\}$. Now, due to the following fundamental inequality
\begin{align}\label{Fund}
|\gamma_1-\gamma_2|^p \le \varepsilon |\gamma_1|^p + C(p,\varepsilon) (|\gamma_1| + |\gamma_2|)^{p-2}|\gamma_1-\gamma_2|^2,
\end{align}
for every $\gamma_1, \gamma_2 \in \mathbb{R}^n$ and $\varepsilon > 0$, we deduce that
\begin{align}\nonumber
\int_{B} |\nabla((u_1- \psi_2)^+)|^p dx  & = \int_{D} |\nabla(u_1- \psi_2)|^p dx \\ \nonumber
 & \le \varepsilon  \int_{D} \left(|\nabla u_1|^p + |\nabla \psi_2|^p \right) dx \\ \nonumber 
& \hspace{1cm} + C \int_{D} (|\nabla u_1| + |\nabla \psi_2|)^{p-2}|\nabla u_1 - \nabla \psi_2|^2 dx \\ \label{est:DOP-4}
& \le \varepsilon \int_{D} \left(|\nabla u_1|^p + |\nabla \psi_2|^p \right) dx.
\end{align}
It is noticeable here that the last estimate comes from~\eqref{est:DOP-3}. Letting $\varepsilon \searrow 0$ in~\eqref{est:DOP-4}, one has $u_1 \le \psi_2$ a.e. in $\Omega_{B}$. It allows us to extend $u_1$ to $\Omega \setminus \Omega_{B}$ by $u$ such that $\psi_1 \le u_1 \le \psi_2$ a.e. and $u_1-u = 0$ in $\Omega \setminus \Omega_{B}$. Taking $\varphi = u_1$ in~\eqref{eq:DOP} and plugging to~\eqref{eq:DOP-1} with $\varphi = u$, it leads to
\begin{align}\nonumber
\int_{\Omega_{B}} & \left\langle \mathcal{A}(x,\nabla u) - \mathcal{A}(x,\nabla u_1), \nabla (u - u_1) \right\rangle dx \le  \int_{\Omega_{B}} \left\langle \mathcal{B}(x,\mathbf{F}), \nabla( u - u_1) \right\rangle dx    \\ \nonumber 
& \hspace{4.5cm} - \int_{\Omega_{B}} \left\langle \mathcal{A}(x,\nabla \psi_2), \nabla ( u - u_1) \right\rangle dx  + \int_{\Omega_{B}} g(u - u_1) dx.
\end{align}
Taking into account basic assumptions~\eqref{eq:B-DOP}, \eqref{eq:A1-DOP} and \eqref{eq:A2-DOP}, it enables us to obtain
\begin{align}\nonumber 
 \int_{\Omega_{B}} & (|\nabla u| + |\nabla u_1|)^{p-2}|\nabla u - \nabla u_1|^2 dx  \le C(L) \left(\int_{\Omega_{B}} |\nabla (u - u_1)| |\mathbf{F}|^{p-1} dx  \right. \\ \label{est:DOP-6} 
 & \hspace{4cm} \left.  + \int_{\Omega_{B}} |\nabla (u - u_1)|  |\nabla \psi_2|^{p-1} dx + \int_{\Omega_{B}} |g||u - u_1| dx\right).
\end{align}
Since $u - u_1 \in W_0^{1,p}(\Omega_B)$, we are able to apply Sobolev's inequality to find out that
\begin{align*}
\int_{\Omega_{B}} |u-u_1|^{p}dx \le C\int_{\Omega_{B}} |\nabla u - \nabla u_1|^{p} dx,
\end{align*}
and together with H\"older and Young's inequalities, one guarantees that 
\begin{align}\label{ineq:S_2}
\int_{\Omega_{B}} |g||u - u_1| dx & \le  \frac{\varepsilon_1}{3}  \int_{\Omega_{B}}  |\nabla u - \nabla u_1|^p  dx +  C(p,\varepsilon_1) \int_{\Omega_{B}} |g|^{\frac{p}{p-1}} dx,
\end{align} 
for every $\varepsilon_1 >0$. We apply again H{\"o}lder and Young's inequalities for two remain terms to discover from~\eqref{est:DOP-6} and~\eqref{ineq:S_2} that
\begin{align}\nonumber
\int_{\Omega_{B}}  (|\nabla u| + |\nabla u_1|)^{p-2}|\nabla u - \nabla u_1|^2 dx & \le \varepsilon_1 \int_{\Omega_{B}} |\nabla u - \nabla u_1|^p dx \\ \label{est:DOP-7}
 & \hspace{2cm} + C(L,p,\varepsilon_1) \int_{\Omega_{B}} |\mathbb{F}|^p dx.
\end{align}
For every $\varepsilon \in (0,1)$ let us apply~\eqref{Fund} to have
\begin{align}\nonumber
\int_{\Omega_{B}} |\nabla u - \nabla u_1|^p dx &\le  \frac{\varepsilon}{2}  \int_{\Omega_{B}} |\nabla u|^p dx + C(p,\varepsilon) \int_{\Omega_{B}} (|\nabla u| + |\nabla u_1|)^{p-2}|\nabla u - \nabla u_1|^2 dx \\ \nonumber
& \le \frac{\varepsilon}{2}   \int_{\Omega_{B}} |\nabla u|^p dx + \varepsilon_1 C(p,\varepsilon)  \int_{\Omega_{B}} |\nabla u - \nabla u_1|^p dx \\ \label{est:DOP-8a}
& \hspace{2cm}  + C(L,p,\varepsilon_1,\varepsilon) \int_{\Omega_{B}} |\mathbb{F}|^p dx,
\end{align}
in which the last estimate comes from~\eqref{est:DOP-7}. It is very easy to take a suitable value of $\varepsilon_1$ depending $\varepsilon$ in~\eqref{est:DOP-8a} to arrive 
\begin{align}\label{est:DOP-8}
\int_{\Omega_{B}} |\nabla u - \nabla u_1|^p dx &\le  \varepsilon  \int_{\Omega_{B}} |\nabla u|^p dx  + C(L,p,\varepsilon) \int_{\Omega_{B}} |\mathbb{F}|^p dx.
\end{align}

{\it The second step: connection to the below constraint.} Next, let us consider $u_2$ as the solution to the equations
\begin{align}\label{eq:omega-2}
\mathcal{L}(u_2)  = \mathcal{L}(\psi_1)  \ \mbox{ in } \ \Omega_{B}, \mbox{ and } \  u_2  =  u_1  \ \mbox{ on } \partial \Omega_{B}. 
\end{align}
Since $u_2 = u_1 \ge \psi_1$ a.e. on $\partial \Omega_{B}$ so it deduces that $u_2 \ge \psi_1$ a.e. in $\Omega_{B}$ by proceeding the same method at the beginning of the proof. So we can take $\varphi = u_2$ in~\eqref{eq:DOP-1} to get that
\begin{align*}
\int_{\Omega_{B}} \left\langle \mathcal{A}(x,\nabla u_1), \nabla u_1 -  \nabla u_2 \right \rangle dx \le \int_{\Omega_{B}} \left \langle \mathcal{A}(x,\nabla \psi_2), \nabla u_1 -  \nabla u_2 \right \rangle dx.
\end{align*}
We then combine with testing the variational formula of~\eqref{eq:omega-2} by $u_1-u_2$ to point out
\begin{align}\nonumber
\int_{\Omega_{B}} & \left\langle \mathcal{A}(x,\nabla u_1) - \mathcal{A}(x,\nabla u_2), \nabla (u_1 - u_2)\right\rangle dx \\ \label{est:DOP-9}
& \hspace{4cm}  = \int_{\Omega_{B}}  \left\langle \mathcal{A}(x,\nabla \psi_2) - \mathcal{A}(x,\nabla \psi_1), \nabla (u_1 - u_2)\right\rangle dx.
\end{align}
The similar technique as the proof of~\eqref{est:DOP-8} will be used again to obtain from~\eqref{est:DOP-9} that
\begin{align}\label{est:DOP-10}
\int_{\Omega_{B}} |\nabla u_1 - \nabla u_2|^p dx \le \varepsilon  \int_{\Omega_{B}} |\nabla u_1|^p dx + C(L,p,\varepsilon) \int_{\Omega_{B}} \left(|\nabla \psi_1|^p + |\nabla \psi_2|^p\right) dx.
\end{align}

{\it The third step: comparison with the homogeneous equation.} Let us now define by $v$  the solution to the homogeneous problem
\begin{align}\label{eq:v-local-a}
\mathcal{L}(v)  = \  0  \ \mbox{ in } \ \Omega_{B}, \ \mbox{ and } \ v  =  u_2 \ \mbox{ on } \ \partial \Omega_{B}. 
\end{align}
Since $u_2 = u_1 = u$ on $\partial \Omega_{B}$ so the problem~\eqref{eq:v-local-a} is exactly~\eqref{eq:v-local}. We can obtain that
\begin{align}\label{est:DOP-11}
\int_{\Omega_{B}} |\nabla u_2 - \nabla v|^p dx \le \varepsilon  \int_{\Omega_{B}} |\nabla v|^p dx + C(L,p,\varepsilon) \int_{\Omega_{B}} |\nabla \psi_1|^p dx.
\end{align}
Finally, let us combine all estimates in~\eqref{est:DOP-8}, \eqref{est:DOP-10} and~\eqref{est:DOP-11} to conclude~\eqref{est:u-v-DOP}, with the fact that both terms $|\nabla u_1|^p$ and $|\nabla v|^p$ can be controlled by $|\nabla u|^p$ and the $\texttt{data}$.
\end{proof}

We next consider another homogeneous equation regarding to the average of $\mathcal{A}$ over the ball $B^*_{2\varrho}$ in $\mathbb{R}^{n-1}$, whenever $B_{2\varrho} \subset B$. The interesting character of the solution to this homogeneous problem is that its gradient still satisfies a type of reverse H{\"o}lder inequality. Moreover, we can establish the local interior difference between gradients of $v$ and the solution $V$ of this equation via the following lemma.

\begin{lemma}\label{lem:RH}
Consider $v$ as a solution to~\eqref{eq:v-local} with the ball $B \subset \Omega$ and consider a ball $B_{2\varrho} \subset B$ for $\varrho>0$. Assume that $V \in v + W^{1,p}_0(B_{2\varrho})$ solves the following problem 
\begin{align}\label{eq:w}
 -\mathrm{div} (\overline{\mathcal{A}}_{B^*_{2\varrho}}(x_1,\nabla V))  = \ 0,  \  \mbox{ in } \ B_{2\varrho}, \ \mbox{ and } \  V  = \ v, \ \mbox{ on } \ \partial B_{2\varrho}.
\end{align}
Then for every $\gamma \ge 1$ there holds
\begin{align}\label{est:RH}
\left(\fint_{B_{\varrho}} |\nabla V|^{\gamma p} dx\right)^{\frac{1}{\gamma}} \le  C\fint_{B_{2\varrho}}|\nabla V|^pdx.
\end{align}
Moreover if $[\mathcal{A}]^{1,2\varrho} \le \delta$ then one has 
\begin{align}\label{est:v-w}
\fint_{B_{\varrho}} |\nabla v - \nabla V|^p dx \le  C \delta \fint_{B_{2\varrho}} |\nabla v|^p dx.
\end{align}
\end{lemma}
\begin{proof}
Let us first refer to~\cite[Theorem 2.1]{Kim2018} for the proof of~\eqref{est:RH}. In order to prove~\eqref{est:v-w}, we will test the variational formulas of~\eqref{eq:v-local} and~\eqref{eq:w} by $v - V \in W_0^{1,p}(B_{2\varrho})$. One obtains that
\begin{align}\nonumber
& \fint_{B_{2\varrho}} \langle \overline{\mathcal{A}}_{B^*_{2\varrho}}(x_1, \nabla v) - \overline{\mathcal{A}}_{B^*_{2\varrho}}(x_1, \nabla V), \nabla v - \nabla V \rangle dx \\ \nonumber 
& \hspace{4cm} = \fint_{B_{2\varrho}} \langle \overline{\mathcal{A}}_{B^*_{2\varrho}}(x_1, \nabla V) - {\mathcal{A}}(x_1, \nabla V), \nabla v - \nabla V \rangle dx,
\end{align}
which implies from~\eqref{eq:A1-DOP}-\eqref{eq:A2-DOP} and the definition of $\theta_1$ in~\eqref{def:theta} that
\begin{align}\label{est:vw-1}
\fint_{B_{2\varrho}} (|\nabla V| + |\nabla v|)^{p-2}|\nabla V-\nabla v|^2 dx \le L \fint_{B_{2\varrho}} |\theta_1\left(\mathcal{A},B_{2\varrho} \right)| |\nabla V|^{p-1} |\nabla v - \nabla V| dx.
\end{align}
Thanks to H{\"o}lder and Young's inequalities, we deduce from~\eqref{est:vw-1} that
\begin{align}\nonumber
\fint_{B_{2\varrho}} (|\nabla V| + |\nabla v|)^{p-2}|\nabla V-\nabla v|^2 dx & \le \varepsilon_1 \fint_{B_{2\varrho}}  |\nabla V - \nabla v|^p dx \\ \label{est:vw-2}
& \hspace{0.2cm}  + C(p,L,\varepsilon_1) \fint_{B_{2\varrho}} |\theta_1\left(\mathcal{A},B_{2\varrho} \right)|^{\frac{p}{p-1}} |\nabla V|^{p} dx,
\end{align}
for any $\varepsilon_1>0$. Moreover, we note that condition~\eqref{eq:A1-DOP} ensures that $|\theta_1\left(\mathcal{A},B_{2\varrho} \right)| \le 2L$. Therefore, thanks to H{\"o}lder's inequality and~\eqref{est:RH} with assumption $[\mathcal{A}]^{1,2\varrho} \le \delta$, for every $\epsilon>0$ one gets that
\begin{align}\nonumber
 \fint_{B_{2\varrho}} |\theta_1\left(\mathcal{A},B_{2\varrho} \right)|^{\frac{p}{p-1}} |\nabla V|^{p} dx  & \le  \left(\fint_{B_{2\varrho}} |\theta_1\left(\mathcal{A},B_{2\varrho} \right)|^{\frac{(1+\epsilon)p}{p-1}} dx\right)^{\frac{1}{1+\epsilon}}\left(\fint_{B_{2\varrho}} |\nabla V|^{\frac{(1+\epsilon)p}{\epsilon}} dx \right)^{\frac{\epsilon}{1+\epsilon}} \\ \nonumber
 & \le L^{\frac{p\epsilon+1}{(p-1)(1+\epsilon)}}\left(\fint_{B_{2\varrho}} |\theta_1\left(\mathcal{A},B_{2\varrho} \right)| dx\right)^{\frac{1}{1+\epsilon}}\left(\fint_{B_{2\varrho}} |\nabla V|^{p} dx \right) \\ \label{est:vw-3}
 & \le L^{\frac{p\epsilon+1}{(p-1)(1+\epsilon)}} \delta^{\frac{1}{1+\epsilon}}\fint_{B_{2\varrho}} |\nabla V|^{p} dx .
\end{align}
Passing $\epsilon$ to $0$ in~\eqref{est:vw-3}, we have
\begin{align}\label{est:vw-4}
\fint_{B_{2\varrho}} |\theta_1\left(\mathcal{A},B_{2\varrho} \right)|^{\frac{p}{p-1}} |\nabla V|^{p} dx \le L^{\frac{1}{p-1}} \delta \fint_{B_{2\varrho}} |\nabla V|^{p} dx \le C(p,L) \delta \fint_{B_{2\varrho}} |\nabla v|^{p} dx.
\end{align}
Substituting~\eqref{est:vw-4} into~\eqref{est:vw-2}, we conclude that
\begin{align}\nonumber 
\fint_{B_{2\varrho}} (|\nabla V| + |\nabla v|)^{p-2}|\nabla V-\nabla v|^2 dx & \le \varepsilon_1 \fint_{B_{2\varrho}}  |\nabla V - \nabla v|^p dx \\ \label{est:vw-5}
& \hspace{1cm}  + C(p,L,\varepsilon_1) \delta \fint_{B_{2\varrho}} |\nabla v|^{p} dx.
\end{align}
Moreover, the fundamental inequality~\eqref{Fund} gives us
\begin{align}\label{est:vw-6}
|\nabla V - \nabla v|^p \le \varepsilon_2 (|\nabla V| + |\nabla v|)^p + C(p,\varepsilon_2) (|\nabla V| + |\nabla v|)^{p-2}|\nabla V - \nabla v|^2,  
\end{align}
for every $\varepsilon_2>0$. Combining between~\eqref{est:vw-5} and~\eqref{est:vw-6}, it arrives to
\begin{align}\nonumber
\fint_{B_{2\varrho}}  |\nabla V - \nabla v|^p dx & \le \varepsilon_2 \fint_{B_{2\varrho}} (|\nabla V|+|\nabla v|)^p dx + \varepsilon_1 C(p,\varepsilon_2)  \fint_{B_{2\varrho}}  |\nabla V - \nabla v|^p dx \\ \nonumber
& \hspace{1cm} + C(p,L,\varepsilon_1,\varepsilon_2) \delta \fint_{B_{2\varrho}} |\nabla v|^{p} dx \\ \nonumber
 & \le  \left[ 4^p \varepsilon_2 + C(p,\varepsilon_2) \varepsilon_1 \right] \fint_{B_{2\varrho}}  |\nabla V - \nabla v|^p dx \\ \label{est:vw-7}
& \hspace{1cm} +  \left[4^p \varepsilon_2 + C(p,L,\varepsilon_1,\varepsilon_2) \delta \right] \fint_{B_{2\varrho}} |\nabla v|^p dx.
\end{align}
Finally, by taking $\varepsilon_2 \in (0, 4^{-p-1})$ satisfying
\begin{align*}
\varepsilon_1 = 4^p \varepsilon_2 [C(p,\varepsilon_2)]^{-1} \ \mbox{ and } \  \varepsilon_2  = 4^{-p}C(p,L,\varepsilon_1,\varepsilon_2) \delta,
\end{align*}
we may conclude~\eqref{est:v-w} from~\eqref{est:vw-7}.
\end{proof}

In order to obtain the comparison estimates near the boundary, we need an additional assumption on $\partial \Omega$ related to Reifenberg flatness condition (this hypothesis exhibits a very low level of regularity). The next lemma also plays a useful tool to verify the boundary version of comparison estimates. Similar to the above argument as in the previous Lemma \ref{lem:RH}, with $\Omega$ is $(r_0,\delta)$-Reifenberg flatness, we also conclude the comparison result on the boundary. The analogous proof technique can be found in several articles such as~\cite{BW2, MP12, MPTNsub}.
\begin{lemma}\label{lem:RH-b}
Let $v$ be a solution to~\eqref{eq:v-local} and consider $\Omega_{2\varrho} :=  B_{2\varrho} \cap \Omega \subset \Omega_{B}$ for some $\varrho>0$. Assume that $V \in v + W^{1,p}_0(\Omega_{2\varrho})$ solves the following problem
\begin{align}\label{eq:w-b}
-\mathrm{div} (\overline{\mathcal{A}}_{\Omega^*_{2\varrho}}(x_1,\nabla V))  = \ 0,  \  \mbox{ in } \ \Omega_{2\varrho}, \ \mbox{ and } \  V  = \ v, \ \mbox{ on } \ \partial \Omega_{2\varrho}.
\end{align}
If $(\mathcal{A},\Omega)$ satisfies assumption $(\mathbb{H})^{r_0,\delta}$ then there holds
\begin{align}\label{est:RH-b}
\left(\fint_{\Omega_{\varrho}} |\nabla V|^{\gamma p} dx\right)^{\frac{1}{\gamma}} \le  C\fint_{\Omega_{2\varrho}}|\nabla V|^pdx,
\end{align}
for every $\gamma \ge 1$ and
\begin{align}\label{est:v-w-b}
\fint_{\Omega_{\varrho}} |\nabla v - \nabla V|^p dx \le  C \delta \fint_{\Omega_{2\varrho}} |\nabla v|^p dx.
\end{align}
\end{lemma}

\section{Weighted level-set approaches}
\label{sec:level-set}

The idea of our approach in this paper is to take advantages of weighted fractional maximal distributions to establish the ``good-$\lambda$'' level-set inequalities. Therefore, the purpose of this section is to give some inequalities associated with the WFMDs. It is worth emphasizing that the construction of these inequalities is the key technique to prove global regularity estimates in the spirit of WFMDs.

Given $\omega \in \mathbf{A}_{\infty}$,  $\xi \in \Omega$ and $\varrho>0$. In what follows,   for $f \in L^1_{\mathrm{loc}}(\Omega)$ we will define the measurable set $\mathcal{W}_{f}(\lambda; B_{\varrho}(\xi))$ as follows
\begin{align}\label{def:Wf}
\mathcal{W}_{f}(\lambda; B_{\varrho}(\xi)) := \left\{x \in \Omega: \ |f(x)|> \lambda\right\} \cap B_{\varrho}(\xi). 
\end{align}
For simplicity of notation, when the open ball $B_{\varrho}(\xi)$ contains $\Omega$ we will write $\mathcal{W}_{f}(\lambda)$ instead of $\mathcal{W}_{f}(\lambda; B_{\varrho}(\xi))$. Moreover, we remind that the distribution function $\mathbf{D}^{\omega}_{f}$ mentioned in this section is defined as in~\eqref{eq:def-dG}.

\begin{lemma}\label{lem:LevSet-1}
For every $\varepsilon>0$ and $a>0$, one can find $\sigma = \sigma(\varepsilon,a)>0$ such that if there exists $\xi_1 \in \Omega$ satisfying ${\mathbf{M}}_{\alpha}(|\mathbb{F}|^p)(\xi_1) \le \sigma\lambda$ for some $\lambda>0$ then there holds
\begin{align}\label{eq:LS-1}
\mathbf{D}^{\omega}_{\mathbf{M}_{\alpha}(|\nabla u|^p)}(\varepsilon^{-a}\lambda) \le \varepsilon \omega(B_{r_0}).
\end{align}
\end{lemma}
\begin{proof}
Thanks to Lemma~\ref{lem:bound-M} and Lemma~\ref{lem:u-v-DOP} with $B \supset \Omega$ and $v \equiv 0$, from the definition of the set $\mathcal{W}_{\mathbf{M}_{\alpha}}$ in~\eqref{def:Wf} it gives us
\begin{align}\label{est-3.1}
\left|\mathcal{W}_{\mathbf{M}_{\alpha}(|\nabla u|^p)}(\varepsilon^{-{a}}\lambda)\right| \le \left(\frac{C}{\varepsilon^{-{a}} \lambda}\int_{\Omega}{|\nabla u|^p dx}\right)^{\frac{n}{n-\alpha}} \le \left(\frac{C}{\varepsilon^{-{a}} \lambda}\int_{\Omega}{|\mathbb{F}|^p dx}\right)^{\frac{n}{n-\alpha}}.
\end{align}
Recall that $\xi_1 \in \Omega$ satisfying ${\mathbf{M}}_{\alpha}(|\mathbb{F}|^p)(\xi_1) \le \sigma\lambda$ (the value of $\sigma$ will be clarified later), it enables us to cover $\Omega$ by an open ball centered at $\xi_1 \in \Omega$ and radius $r = 2\mathrm{diam}(\Omega)$, it leads to
\begin{align}\label{est:LS-1}
\int_{\Omega} |\mathbb{F}|^p dx \le C r^{n-\alpha} \left(r^{\alpha} \fint_{B_r(\xi_1)} |\mathbb{F}|^p dx \right) \le  C r^{n-\alpha}  {\mathbf{M}}_{\alpha}(|\mathbb{F}|^p)(\xi_1) \le C r^{n-\alpha} \sigma \lambda.
\end{align}
Substituting~\eqref{est:LS-1} into~\eqref{est-3.1}, there holds
\begin{align}\nonumber
\left|\mathcal{W}_{\mathbf{M}_{\alpha}(|\nabla u|^p)}(\varepsilon^{-{a}}\lambda)\right|  \le C (\sigma\varepsilon^a)^{\frac{n}{n-\alpha}} r^n \le C (\mathrm{diam}(\Omega)/r_0)^n (\sigma\varepsilon^a)^{\frac{n}{n-\alpha}}  |B_{r_0}|,
\end{align}
which implies from the definitions of Muckenhoupt weight $\omega$ and function $\mathbf{D}^{\omega}_{\mathbf{M}_{\alpha}(|\nabla u|^p)}$ that
\begin{align}\nonumber
\mathbf{D}^{\omega}_{\mathbf{M}_{\alpha}(|\nabla u|^p)}(\varepsilon^{-{a}}\lambda) & \le c_0\left(\frac{\left|\mathcal{W}_{\mathbf{M}_{\alpha}(|\nabla u|^p)}(\varepsilon^{-{a}}\lambda)\right|}{|B_{r_0}|}\right)^{\nu} \omega(B_{r_0}) \\ \label{est:LS-2}
& \le C (\mathrm{diam}(\Omega)/r_0)^{n\nu} (\sigma\varepsilon^a)^{\frac{n \nu}{n-\alpha}}  \omega(B_{r_0}).
\end{align}
Let us take $\sigma$ depending on $\varepsilon$, $a$ and $\texttt{data}$ in~\eqref{est:LS-2} such that
\begin{align}\label{est:beta}
0 < C (\mathrm{diam}(\Omega)/r_0)^{n\nu} (\sigma\varepsilon^a)^{\frac{n \nu}{n-\alpha}} <  \varepsilon,
\end{align}
to conclude~\eqref{eq:LS-1} and finish the proof.
\end{proof}

\begin{lemma}\label{lem:LevSet-2}
Let $a>0$ and $\xi_2 \in B_{\varrho}(\xi)$ satisfying $\mathbf{M}_{\alpha}(|\nabla u|^p)(\xi_2)  \le \lambda$. Then one can find $\xi_0 \in \overline{\Omega}$ and $k \in \mathbb{N}$ such that the following inequality
\begin{align}\label{eq:LS-2}
\mathbf{D}^{\omega}_{\mathbf{M}_{\alpha}(|\nabla u|^p)}(\varepsilon^{-{a}}\lambda; B_{\varrho}(\xi)) \le \mathbf{D}^{\omega}_{\mathbf{M}_{\alpha}(\chi_{B_{k\varrho}(\xi_0)}|\nabla u|^p)}(\varepsilon^{-{a}}\lambda; B_{\varrho}(\xi)),
\end{align}
holds for every $\varepsilon \in \left(0,3^{-\frac{n}{a}}\right)$.
\end{lemma}
\begin{proof}
For every $\zeta \in B_{\varrho}(\xi)$, it is easy to check that $B_{r}(\zeta) \subset B_{3r}(\xi_2)$ for all $r \ge \varrho$, which allows us take into account assumption ${\mathbf{M}}_{\alpha}(|\nabla u|^p)(\xi_2)  \le \lambda$ to find
\begin{align*}
\sup_{r \ge \varrho} r^{\alpha} {\fint_{B_{r}(\zeta)}{|\nabla u|^p dx}} \le 3^n\sup_{r \ge \varrho}{ r^{\alpha} \fint_{B_{3r}(\xi_2)}{|\nabla u|^p dx}} \le 3^{n-\alpha} {\mathbf{M}}_{\alpha}(|\nabla u|^p)(\xi_2) \le 3^n \lambda.
\end{align*}
Therefore we may conclude that 
$${\mathbf{M}}_{\alpha}(|\nabla u|^p)(\zeta)  \le \max \left\{  \sup_{0< r < \varrho}{r^{\alpha} \fint_{B_{r}(\zeta)}{|\nabla u|^p dx}} ; \  3^n \lambda \right\}, \quad \mbox{ for all } \zeta \in B_{\varrho}(\xi).$$ 
If we choose $\varepsilon_0 = 3^{-\frac{n}{a}}$ then for every $\varepsilon \in (0,\varepsilon_0)$, one has
\begin{align}\label{est:res11}
\mathcal{W}_{\mathbf{M}_{\alpha}(|\nabla u|^p)}(\varepsilon^{-{a}}\lambda; B_{\varrho}(\xi)) = \left\{ \sup_{0< r < \varrho}{r^{\alpha} \fint_{B_{r}(\zeta)}{|\nabla u|^p dx}} > \varepsilon^{-{a}}\lambda \right\} \cap B_{\varrho}(\xi).
\end{align}
If $B_{8\varrho}(\xi) \subset\Omega$ let us take $\xi_0 = \xi$ and $k = 2$. Otherwise, if $B_{8\varrho}(\xi) \cap \partial\Omega \neq \emptyset$ one can find $\xi_0 \in \partial \Omega$ such that $|\xi_0 - \xi| = \mathrm{dist}(\xi,\partial \Omega) \le 8\varrho$, then we choose $k = 16$. With this choice, one can see that 
$$B_{r}(\zeta) \subset B_{2\varrho}(\xi) \subset B_{k\varrho}(\xi_0), \quad \mbox{ for any } \ 0 < r < \varrho \  \mbox{ and } \ \zeta \in B_{\varrho}(\xi).$$ 
For this reason, we only need to replace the integral $\fint_{B_{r}(\zeta)}{|\nabla u|^p dx}$ in~\eqref{est:res11} by the other one $\fint_{B_{r}(\zeta)}{\chi_{B_{k\varrho}(\xi_0)}|\nabla u|^p dx}$ in order to obtain~\eqref{eq:LS-2} from~\eqref{est:res11}.
\end{proof}

\begin{lemma}\label{lem:LevSet-3}
For every  $0 < a < \frac{2}{\nu}\left(1 - \frac{\alpha}{n}\right)$, one can find a constant $\varepsilon_0 = \varepsilon_0(\texttt{data}) \in \left(0, 3^{-\frac{n}{a}}\right)$ and numbers $\sigma = \sigma(a,\varepsilon)>0$, $\delta = \delta(a,\varepsilon)$ such that if $(\mathcal{A},\Omega)$ satisfies assumption $(\mathbb{H})^{r_0,\delta}$ and there are $\xi_2, \, \xi_3 \in B_{\varrho}(\xi)$ satisfying
\begin{align}\label{eq:x2}
{\mathbf{M}}_{\alpha}(|\nabla u|^p)(\xi_2) \le \lambda \ \mbox{ and } \
{\mathbf{M}}_{\alpha}(|\mathbb{F}|^p)(\xi_3) \le \sigma\lambda,
\end{align}
for some $\lambda>0$ then for every $\varepsilon \in (0,\varepsilon_0)$ there holds
\begin{align}\label{eq:iigoal}
\mathbf{D}^{\omega}_{\mathbf{M}_{\alpha}(|\nabla u|^p)}(\varepsilon^{-{a}}\lambda; B_{\varrho}(\xi)) < \varepsilon \omega(B_{\varrho}(\xi)).
\end{align} 
\end{lemma}
\begin{proof}
Lemma~\ref{lem:LevSet-2} under condition~\eqref{eq:x2} gives us the existence of $\xi_0 \in \overline{\Omega}$ and $k \le 16$ such that~\eqref{eq:LS-2} holds for every $\varepsilon  \in \left(0, 3^{-\frac{n}{a}}\right)$. We define $B_{i} = B_{2^{i-1}k\varrho}(\xi_0)$ for $i = 1,2,3$. Suppose that $v \in u + W_0^{1,p}(B_3)$ solves the  problem
\begin{equation}\nonumber 
 \mathcal{L}(v)  =  0 \ \mbox{ in } B_3 \cap \Omega, \ \mbox { and } \ v  =  u  \ \mbox{ on } \partial (B_3 \cap \Omega).
\end{equation}
Lemma~\ref{lem:u-v-DOP} gives us the comparison estimate between $\nabla v$ and $\nabla u$ as below
\begin{align}\label{est:100}
\fint_{B_3} |\nabla u - \nabla v|^p dx \le \varepsilon_1  \fint_{B_3} |\nabla u|^p dx +  C(\varepsilon_1) \fint_{B_3} |\mathbb{F}|^p dx,
\end{align}
for any $\varepsilon_1>0$. By setting of $B_3$, it is possible to claim that 
$$B_{2\varrho}(\xi) \subset B_1 = B_{k\varrho}(\xi_0) \subset B_3 \subset B_{64\varrho}(\xi_0) \subset B_{72\varrho}(\xi) \subset B_{73\varrho}(\xi_2) \cap B_{73\varrho}(\xi_3),$$ 
and of course $|B_3| \sim \varrho^n$. Combining with~\eqref{eq:x2}, one has
\begin{align}\label{eq:com-2}
\fint_{B_3}{|\nabla u|^pdx} \le \frac{|B_{73\varrho}(\xi_2)|}{|B_3|} \fint_{B_{73\varrho}(\xi_2)}{|\nabla u|^pdx} \le C \varrho^{-\alpha} {\mathbf{M}}_{\alpha}(|\nabla u|^p)(\xi_2) \le C \lambda \varrho^{-\alpha},
\end{align}
and similarly
\begin{align}\label{eq:com-22}
\fint_{B_3}{|\mathbb{F}|^pdx} \le \frac{|B_{73\varrho}(\xi_3)|}{|B_3|} \fint_{B_{73\varrho}(\xi_3)}{|\mathbb{F}|^pdx} \le C \varrho^{-\alpha} {\mathbf{M}}_{\alpha}(|\mathbb{F}|^p)(\xi_3) \le C \sigma \lambda \varrho^{-\alpha}.
\end{align}
Substituting~\eqref{eq:com-2} and~\eqref{eq:com-22} into~\eqref{est:100}, we find
\begin{align}\label{eq:u-v-B3}
\fint_{B_3} |\nabla u - \nabla v|^p dx \le C \left[\varepsilon_1 +  C(\varepsilon_1) \sigma\right]\lambda \varrho^{-\alpha}.
\end{align}
Let us now consider $V \in v + W_0^{1,p}(B_2)$ solving the next problem
\begin{equation}\nonumber 
\begin{cases} -\mbox{div} (\overline{\mathcal{A}}_{B^*_1}(x_1,\nabla V)) & = \ 0, \quad  \quad \mbox{ in } B_2 \cap \Omega,\\ 
\hspace{1.2cm} V & = \ v, \qquad \mbox{ on } \partial (B_2 \cap \Omega).\end{cases}
\end{equation}
Lemma~\ref{lem:RH} and~\ref{lem:RH-b} state that if $(\mathcal{A},\Omega)$ satisfies assumption $(\mathbb{H})^{r_0,\delta}$ then $\nabla V$ satisfies the following reverse H{\"o}lder's inequality
\begin{align}\label{est:101}
\left(\fint_{B_1} |\nabla V|^{\gamma p} dx\right)^{\frac{1}{\gamma}} \le  C\fint_{B_2}|\nabla V|^pdx, \quad \mbox{ for all } \gamma \ge 1,
\end{align}
and the comparison estimate with $\nabla v$ as below
\begin{align}\label{est:102}
\fint_{B_1} |\nabla v - \nabla V|^p dx \le  C \delta \fint_{B_2} |\nabla v|^p dx.
\end{align}
On the other hand, from inequality~\eqref{eq:LS-2} in Lemma~\ref{lem:LevSet-2} and the definition of Muckenhoupt weight $\omega \in \mathbf{A}_{\infty}$, there holds
\begin{align}\nonumber
\mathbf{D}^{\omega}_{\mathbf{M}_{\alpha}(|\nabla u|^p)}(\varepsilon^{-{a}}\lambda; B_{\varrho}(\xi)) & \le \mathbf{D}^{\omega}_{\mathbf{M}_{\alpha}(\chi_{B_1}|\nabla u|^p)}(\varepsilon^{-{a}}\lambda; B_{\varrho}(\xi)) \\ \label{est:301}
& \le c_0 \left[\frac{\left|\mathcal{W}_{\mathbf{M}_{\alpha}(\chi_{B_1}|\nabla u |^p)}( \varepsilon^{-{a}}\lambda; B_{\varrho}(\xi)) \right|}{|B_{\varrho}(\xi)|}\right]^{\nu} \omega(B_{\varrho}(\xi)),
\end{align}
where $(c_0,\nu) = [\omega]_{\mathbf{A}_{\infty}}$. By using an elementary inequality
one deduces from~\eqref{est:301} that
\begin{align}\label{eq:estV-1}
\mathbf{D}^{\omega}_{\mathbf{M}_{\alpha}(|\nabla u|^p)}(\varepsilon^{-{a}}\lambda; B_{\varrho}(\xi)) \le C \left(\mathrm{I} + \mathrm{II} + \mathrm{III}\right)^{\nu} \varrho^{-n\nu} \omega(B_{\varrho}(\xi)),
\end{align}
where $\mathrm{I}$, $\mathrm{II}$ and $\mathrm{III}$ are given by
\begin{align}\nonumber
& \mathrm{I} := \left|\left\{{\mathbf{M}}_{\alpha}(\chi_{B_1}|\nabla u - \nabla v|^p)> 3^{-p}\varepsilon^{-{a}}\lambda \right\} \right|, \\ \nonumber
& \mathrm{II} := \left|\left\{{\mathbf{M}}_{\alpha}(\chi_{B_1}|\nabla v - \nabla V|^p)> 3^{-p}\varepsilon^{-{a}}\lambda \right\} \right|, \\ \nonumber
& \mathrm{III} := \left|\left\{{\mathbf{M}}_{\alpha}(\chi_{B_1}|\nabla V|^p)> 3^{-p}\varepsilon^{-{a}}\lambda \right\} \right|. 
\end{align}
Thanks to Lemma~\ref{lem:bound-M} with $s = 1$, there holds
\begin{align}\nonumber
\mathrm{I} &\le \left(\frac{C}{3^{-p}\varepsilon^{-{a}}\lambda} \int_{B_1} |\nabla u - \nabla v|^p dx \right)^{\frac{n}{n-\alpha}}  \le \left(\frac{C |B_3|}{\varepsilon^{-{a}}\lambda} \fint_{B_3} |\nabla u - \nabla v|^p dx \right)^{\frac{n}{n-\alpha}},
\end{align}
which with~\eqref{eq:u-v-B3} implies to
\begin{align}\label{est:I}
\mathrm{I}  \le C \left[\varepsilon^{a}(\varepsilon_1 + C(\varepsilon_1) \sigma)\varrho^{n-\alpha}\right]^{\frac{n}{n-\alpha}} \le C \varrho^n \left[\varepsilon^{a} \varepsilon_1  + C(\varepsilon_1) \varepsilon^{a} \sigma \right]^{\frac{n}{n-\alpha}} .
\end{align}
We now apply this argument again to estimate $\mathrm{II}$, by combining Lemma~\ref{lem:bound-M} with $s=1$ and~\eqref{est:102} to arrive
\begin{align*}
\mathrm{II} &\le \left( \frac{C}{3^{-p}\varepsilon^{-{a}}\lambda} \int_{B_1} |\nabla v - \nabla V|^p dx \right)^{\frac{n}{n-\alpha}} \le \left(\frac{C\delta |B_1|}{\varepsilon^{-{a}}\lambda} \fint_{B_2} |\nabla v|^p dx\right)^{\frac{n}{n-\alpha}}.
\end{align*}
Taking into account~\eqref{eq:x2} and~\eqref{eq:u-v-B3}, one has
\begin{align*}
\fint_{B_2} |\nabla v|^p dx \le C\left(\fint_{B_3} |\nabla u|^p dx + \fint_{B_3} |\nabla u - \nabla v|^p dx\right) \le C \left[1 + \varepsilon_1 + C(\varepsilon_1)\sigma \right]\lambda \varrho^{-\alpha}.
\end{align*}
Both previous inequalities give us
\begin{align}\label{est:II}
\mathrm{II} &\le C  \left[\delta \varepsilon^{a} \left(1 + \varepsilon_1 + C(\varepsilon_1)\sigma\right)\right]^{\frac{n}{n-\alpha}} \varrho^n.
\end{align}
For any $\theta>1$, the last term $\mathrm{III}$ can be bounded by using Lemma~\ref{lem:bound-M} with $s = {\theta}>1$ and combining the reverse H{\"o}lder inequality~\eqref{est:101} with $\gamma = {\theta}$ to have
\begin{align}\label{eq:com-3}
\mathrm{III} &\le \left( \frac{C|B_1|}{\left(3^{-p}\varepsilon^{-{a}}\lambda\right)^{{\theta}}} \fint_{B_1} |\nabla V|^{p{\theta}} dx \right)^{\frac{n}{n-\alpha {\theta}}} \le \left[ \frac{C|B_1|}{\left(\varepsilon^{-{a}}\lambda\right)^{{\theta}}} \left( \fint_{B_2} |\nabla V|^{p} dx \right)^{{\theta}}\right]^{\frac{n}{n-\alpha {\theta}}}.
\end{align}
From \eqref{est:102} with $\delta \in (0,1)$, it guarantees that
\begin{align*}
\fint_{B_2} |\nabla V|^{p} dx \le C \left(\fint_{B_2} |\nabla v|^{p} dx  + \fint_{B_2} |\nabla v - \nabla V|^{p} dx\right) \le C \fint_{B_2} |\nabla v|^{p} dx,
\end{align*}
which allows us to arrive the following conclusion by collecting the previous computation 
\begin{align*}
\fint_{B_2} |\nabla V|^{p} dx \le C \left[1 + \varepsilon_1 + C(\varepsilon_1)\sigma\right]\lambda \varrho^{-\alpha}.
\end{align*}
Substituting this estimate into~\eqref{eq:com-3}, one can find
\begin{align}\nonumber
\mathrm{III} & \le C \left[ \frac{\varrho^n}{(\varepsilon^{-a}\lambda)^{{\theta}}} \left(1 + \varepsilon_1 + C(\varepsilon_1)\sigma\right)^{{\theta}} \lambda^{{\theta}} \varrho^{-\alpha {\theta}} \right]^{\frac{n}{n-\alpha {\theta}}} \\ \label{est:III}
& \le C \left[\varepsilon^a\left(1 + \varepsilon_1 + C(\varepsilon_1) \sigma\right)\right]^{\frac{n{\theta}}{n-\alpha {\theta}}} \varrho^n.
\end{align}
Plugging estimations of $\mathrm{I}$, $\mathrm{II}$ and $\mathrm{III}$ from~\eqref{est:I}, \eqref{est:II} and~\eqref{est:III} respectively, one gets from~\eqref{eq:estV-1} that
\begin{align}\nonumber
\mathbf{D}^{\omega}_{\mathbf{M}_{\alpha}(|\nabla u|^p)}(\varepsilon^{-{a}}\lambda; B_{\varrho}(\xi)) & \le C \left\{ \left[\varepsilon^{a}\left(\varepsilon_1 + C(\varepsilon_1) \sigma \right)\right]^{\frac{n}{n-\alpha}} \varrho^n \right. \\ \nonumber
& \qquad \quad \left.  + \left[\delta \varepsilon^{a} \left(1 + \varepsilon_1 + C(\varepsilon_1)\sigma\right)\right]^{\frac{n}{n-\alpha}} \varrho^n \right. \\ \label{est:VI}
& \qquad \quad \left.  + \left[\varepsilon^{a}\left(1 + \varepsilon_1 + C(\varepsilon_1)\sigma\right)\right]^{\frac{n{\theta}}{n-\alpha {\theta}}} \varrho^n\right\}^{\nu} \varrho^{-n\nu} \omega(B_{\varrho}(\xi)).
\end{align}
In the inequality~\eqref{est:VI}, it is possible to choose $\sigma$ satisfying~\eqref{est:beta} and $\varepsilon_1, \delta$ such that
\begin{align*}
(\delta\varepsilon^a)^{\frac{n\nu}{n-\alpha}} < \varepsilon^{2}, \quad \varepsilon_1 = \delta \in (0,1), \quad  0 < \sigma < \varepsilon_1 [C(\varepsilon_1)]^{-1},
\end{align*} 
which guarantees that
\begin{align}\label{est:VII}
\mathbf{D}^{\omega}_{\mathbf{M}_{\alpha}(|\nabla u|^p)}(\varepsilon^{-{a}}\lambda; B_{\varrho}(\xi)) & \le C \left(\varepsilon^{\frac{2}{\nu}} + \varepsilon^{\frac{an{\theta}}{n-\alpha {\theta}}} \right)^{\nu} \omega(B_{\varrho}(\xi)).
\end{align}
The most interesting point here is that assumption  $0 < a < \frac{2}{\nu}\left(1 - \frac{\alpha}{n}\right)$ allows us to take ${\theta} = \frac{2}{a\nu + \frac{2\alpha}{n}}>1$ which reduces to ${\frac{an{\theta}}{n-\alpha {\theta}}} = \frac{2}{\nu}$. With this choice of $\theta$, inequality~\eqref{est:VII} becomes to
\begin{align}\nonumber
\mathbf{D}^{\omega}_{\mathbf{M}_{\alpha}(|\nabla u|^p)}(\varepsilon^{-{a}}\lambda; B_{\varrho}(\xi)) & \le C \varepsilon^2 \omega(B_{\varrho}(\xi)),
\end{align}
and it follows to~\eqref{eq:iigoal} for $\varepsilon$ small enough.  The proof is then complete.
\end{proof}

\section{Proofs of main results}
\label{sec:proofs}

Our strategy now becomes clear and with aid of preliminary lemmas and estimates proved in previous sections, we are ready to prove the main results.

\subsection{Proof of Theorem \ref{theo-A}}
\label{sec:proofA}
\begin{proof}[Proof of Theorem~\ref{theo-A}]
One can see that the inequality~\eqref{est:theo-A} is a sequence of the following inequality
\begin{align}\label{eq:mainlambda} 
\omega\left(\mathcal{W}_{\mathbf{M}_{\alpha}(|\nabla u|^p)}(\varepsilon^{-a}\lambda) \cap \left(\mathcal{W}_{\mathbf{M}_{\alpha}(|\mathbb{F}|^p)}(\sigma\lambda)\right)^c\right)
  \leq C \varepsilon \left(\mathcal{W}_{\mathbf{M}_{\alpha}(|\nabla u|^p)}(\lambda) \right).
\end{align}
Therefore we have just determined $\varepsilon_0 =\varepsilon_0(\texttt{data}) \in (0,1)$ such that: for any $\lambda>0$, $\varepsilon \in (0,\varepsilon_0)$ and $a \in \left(0,\frac{2}{\nu}\left(1-\frac{\alpha}{n}\right)\right)$ we can find $\delta = \delta(a,\varepsilon)>0$ and $\sigma = \sigma(a,\varepsilon)>0$ valid~\eqref{eq:mainlambda}, under assumptions $\Omega$ is $(r_0,\delta)$-Reifenberg and $[\mathcal{A}]^{1,r_0} \le \delta$ for some $r_0>0$. 

Let us rewrite the inequality~\eqref{eq:mainlambda} as $\omega(\mathcal{S}^{\lambda}_{\varepsilon}) \le C \varepsilon \omega (\mathcal{R}^{\lambda})$, where $\mathcal{S}^{\lambda}_{\varepsilon}$ and $\mathcal{R}^{\lambda}$ present the sets appeared on the left and right hand side respectively.  The proof is straightforward from the covering Lemma~\ref{lem:cover-lem} for two sets $\mathcal{S}^{\lambda}_{\varepsilon}$ and $\mathcal{R}^{\lambda}$. Hence we proceed to show that two hypotheses of Lemma~\ref{lem:cover-lem} are satisfied.

Obviously, we can prove~\eqref{eq:mainlambda} with assumption $\mathcal{S}^{\lambda}_{\varepsilon} \neq \emptyset$ which allows us to have $\xi_1 \in \Omega$ such that ${\mathbf{M}}(|\mathbb{F}|^p)(\xi_1) \le \sigma\lambda$. Given $r_0>0$, Lemma~\ref{lem:LevSet-1} gives us a suitable value of $\sigma = \sigma(\varepsilon)$ that is valid the following inequality
\begin{align}\nonumber
\omega(\mathcal{S}^{\lambda}_{\varepsilon}) \le \mathbf{D}^{\omega}_{\mathbf{M}_{\alpha}(|\nabla u|^p)}(\varepsilon^{-a}\lambda) \le \varepsilon \omega(B_{r_0}).
\end{align}

On the other hand, Lemma~\ref{lem:LevSet-3} shows that if there exist $\xi_2 \in B_{\varrho}(\xi)\cap \Omega \cap (\mathcal{R}^{\lambda})^c$ and $\xi_3 \in \mathcal{S}^{\lambda}_{\varepsilon} \cap B_{\varrho}(\xi)$ which deduce to
\begin{align}\nonumber
{\mathbf{M}}(|\nabla u|^p)(\xi_2) \le \lambda \ \mbox{ and } \
{\mathbf{M}}(|\mathbb{F}|^p)(\xi_3) \le \varepsilon^{b}\lambda,
\end{align}
then one can find $\varepsilon_0 \in (0,1)$, $\sigma = \sigma(a,\varepsilon)>0$ and $\delta = \delta(a,\varepsilon)$ such that
\begin{align*}
\omega(\mathcal{S}^{\lambda}_{\varepsilon} \cap B_{\varrho}(\xi)) \le \mathbf{D}^{\omega}_{\mathbf{M}_{\alpha}(|\nabla u|^p)}(\varepsilon^{-{a}}\lambda; B_{\varrho}(\xi)) < \varepsilon \omega(B_{\varrho}(\xi)),
\end{align*}
for all $\varepsilon \in (0,\varepsilon_0)$ provided $\Omega$ is $(r_0,\delta)$-Reifenberg and $[\mathcal{A}]^{1,r_0} \le \delta$. For this reason, the second hypothesis of Lemma~\ref{lem:cover-lem} can be directly obtained by contradiction. 

Therefore, the inequality~\eqref{eq:mainlambda} holds for $0 < \varepsilon < \varepsilon_0$, and the conclusion of weighted distribution inequality~\eqref{est:theo-A} also follows.
\end{proof}

\subsection{Proof of Theorem \ref{theo-B} and Corollary \ref{cor-B}}
\label{sec:proofB}
\begin{proof}[Proof of Theorem \ref{theo-B}]
Here, our attention has been focused on the case of $L^{q,s}_{\omega}(\Omega)$ for $0< q < \infty$ and $0<s < \infty$. The proof for the case $s = \infty$ is also obtained with a slight changing of calculation. \\
Firstly, for any $\alpha \in [0,n)$ and $(\nu,c_0) = [\omega]_{\mathbf{A}_{\infty}}$ let us take
$$0 < a < \min\left\{\frac{2}{\nu}\left(1-\frac{\alpha}{n}\right); \frac{1}{q}\right\}.$$ 
Theorem~\ref{theo-A} ensures the existence of $\varepsilon_0 \in (0,1)$, $\delta>0$ and $\sigma > 0$ such that if assumption $(\mathbb{H})^{r_0,\delta}$ of $(\mathcal{A},\Omega)$ is satisfied, then the inequality~\eqref{est:theo-A} holds for every $0< \varepsilon <\varepsilon_0$ and $\lambda>0$. More precisely, there holds
\begin{align}\label{est:BB-1}
\mathbf{D}^{\omega}_{\mathbf{M}_{\alpha}(|\nabla u|^p)}(\varepsilon^{-a}\lambda) \le C \varepsilon \mathbf{D}^{\omega}_{\mathbf{M}_{\alpha}(|\nabla u|^p)}(\lambda)  + \mathbf{D}^{\omega}_{\mathbf{M}_{\alpha}(|\mathbb{F}|^p)}(\sigma \lambda).
\end{align}
By applying~\eqref{est:BB-1} and performing several times of changing variables, one gets that
\begin{align}\nonumber
\|\mathbf{M}_{\alpha}(|\nabla u|^p)\|^s_{L^{q,s}_{\omega}(\Omega)} & =  \varepsilon^{-as}q\int_0^\infty \lambda^{s-1} \left[\mathbf{D}^{\omega}_{\mathbf{M}_{\alpha}(|\nabla u|^p)}(\varepsilon^{-a}\lambda)\right]^{\frac{s}{q}}  {d\lambda} \\ \nonumber
  & \le C  \varepsilon^{-as}q \int_0^\infty \lambda^{s-1} \left[\varepsilon \mathbf{D}^{\omega}_{\mathbf{M}_{\alpha}(|\nabla u|^p)}(\lambda) \right]^{\frac{s}{q}} {d\lambda} \\ \nonumber
& \hspace{3cm} +  C \varepsilon^{-as} q \int_0^\infty \lambda^{s-1} \left[\mathbf{D}^{\omega}_{\mathbf{M}_{\alpha}(|\mathbb{F}|^p)}(\sigma\lambda)\right]^{\frac{s}{q}} {d\lambda}  \\ \nonumber
& \le  C \varepsilon^{\frac{s}{q}-as} \|\mathbf{M}_{\alpha}(|\nabla u|^p)\|^s_{L^{q,s}_{\omega}(\Omega)} +  C  \sigma^{s} \varepsilon^{-as} \|\mathbf{M}_{\alpha}(|\mathbb{F}|^p)\|^s_{L^{q,s}_{\omega}(\Omega)},
\end{align}
which leads to the following estimate with an elementary inequality
\begin{align}\label{est:5b}
\|\mathbf{M}_{\alpha}(|\nabla u|^p)\|_{L^{q,s}_{\omega}(\Omega)}   & \le  C \varepsilon^{\frac{1}{q}-a} \|\mathbf{M}_{\alpha}(|\nabla u|^p)\|_{L^{q,s}_{\omega}(\Omega)} +  C  \sigma \varepsilon^{-a} \|\mathbf{M}_{\alpha}(|\mathbb{F}|^p)\|_{L^{q,s}_{\omega}(\Omega)}.
\end{align}
Since $\frac{1}{q}- a > 0$ we can fix $\varepsilon$ in~\eqref{est:5b} small enough to conclude~\eqref{est:theo-B}. 

To prove the point-wise estimate~\eqref{est:theo-B-2} related to the Riesz potential defined as in~\eqref{def:Riesz}, we refer to~\cite[Lemma 4.2]{MPTNsub} for the statement: if the following inequality
\begin{align}\label{cond:w}
\int_{\mathbb{R}^n} \varphi(x)  d\omega(x) \le C \int_{\mathbb{R}^n} \psi(x)  d\omega(x), 
\end{align}
holds for any $\omega \in \mathcal{A}_1$ and $\beta \in (0,n)$, then 
\begin{align}\label{eq:wA1}
\mathbf{I}_{\beta}\varphi(x) \le C \mathbf{I}_{\beta}\psi(x), \quad \mbox{ a.e. in } \  \mathbb{R}^n.
\end{align}
A nice feature of the weighted Lorentz space $L^{q,s}_{\omega}(\Omega)$ is that it becomes the weighted Lebesgue space $L^{q}_{\omega}(\Omega)$ in the special case $q=s$. Hence, for $0< t < \infty$ let us apply~\eqref{est:theo-B} for $q=s=t$, one obtains that
\begin{align*}
\int_{\mathbb{R}^n} \chi_{\Omega} |\mathbf{M}_{\alpha}(|\nabla u|^p)|^t  d\omega(z) \le C\int_{\mathbb{R}^n} \chi_{\Omega} |\mathbf{M}_{\alpha}(|\mathbb{F}|^p)|^t d\omega(z), 
\end{align*}
which is valid~\eqref{cond:w} with $\varphi =  \chi_{\Omega} |\mathbf{M}_{\alpha}(|\nabla u|^p)|^t$ and $\psi = \chi_{\Omega} |\mathbf{M}_{\alpha}(|\mathbb{F}|^p)|^t$. Therefore, one deduces to~\eqref{est:theo-B-2} from~\eqref{eq:wA1} directly and the proof is finished. 
\end{proof}

\subsection{Proof of Theorem \ref{theo:L-O}}
\label{sec:proofC}
\begin{proof}
Since $\Phi \in \Delta_2$, it is well-know that one can find a constant $p_1>1$ satisfying 
\begin{align}\label{Phi-p1}
\Phi(t\lambda) \le C t^{p_1} \Phi(\lambda), \qquad \mbox{ for any } t \ge 1 \ \mbox{ and } \ \lambda>0.
\end{align}
For every $0 < q < \infty$ and $0< s  < \infty$, let us choose
\begin{align}\label{choose-a}
0 < a < \min\left\{\frac{2}{\nu}\left(1 - \frac{\alpha}{n}\right); \frac{1}{p_1 q}\right\}.
\end{align}
Thanks to Theorem~\ref{theo-A}, for any $\lambda>0$ and $\varepsilon$ small enough there exist $\delta = \delta(a,\alpha,\varepsilon)>0$ and $\sigma = \sigma(a,\alpha,\varepsilon)>0$ such that if assumption $(\mathbb{H})^{r_0,\delta}$ of $(\mathcal{A},\Omega)$ is satisfied then there holds
\begin{align}\label{est:O-1}
\mathbf{D}^{\omega}_{\mathbf{M}_{\alpha}(|\nabla u|^p)}(\varepsilon^{-a}\lambda) \le C \varepsilon \mathbf{D}^{\omega}_{\mathbf{M}_{\alpha}(|\nabla u|^p)}(\lambda)  + \mathbf{D}^{\omega}_{\mathbf{M}_{\alpha}(|\mathbb{F}|^p)}(\sigma \lambda).
\end{align}
Let us replace $\lambda$ in~\eqref{est:O-1} by $\varepsilon^{a} \Phi^{-1}(\lambda)$, it becomes to
\begin{align}\nonumber
\mathbf{D}^{\omega}_{\mathbf{M}_{\alpha}(|\nabla u|^p)}(\Phi^{-1}(\lambda)) \le C \varepsilon \mathbf{D}^{\omega}_{\mathbf{M}_{\alpha}(|\nabla u|^p)}(\varepsilon^{a}\Phi^{-1}(\lambda))  + \mathbf{D}^{\omega}_{\mathbf{M}_{\alpha}(|\mathbb{F}|^p)}(\sigma \varepsilon^{a}\Phi^{-1}(\lambda)),
\end{align} 
which is equivalent to
\begin{align} \nonumber 
\mathbf{D}^{\omega}_{\Phi(\mathbf{M}_{\alpha}(|\nabla u|^p))}(\lambda) \le  C \varepsilon \mathbf{D}^{\omega}_{\Phi(\varepsilon^{-a}\mathbf{M}_{\alpha}(|\nabla u|^p))}(\lambda)  + \mathbf{D}^{\omega}_{\Phi(\sigma^{-1}\varepsilon^{-a}\mathbf{M}_{\alpha}(|\mathbb{F}|^p))}(\lambda).
\end{align} 
We may~\eqref{Phi-p1} on this inequality to arrive
\begin{align}\label{est:O-3}
\mathbf{D}^{\omega}_{\Phi(\mathbf{M}_{\alpha}(|\nabla u|^p))}(\lambda) \le  C \varepsilon \mathbf{D}^{\omega}_{\Phi(\mathbf{M}_{\alpha}(|\nabla u|^p))}(C\varepsilon^{ap_1}\lambda)  + \mathbf{D}^{\omega}_{\Phi(\mathbf{M}_{\alpha}(|\mathbb{F}|^p))}(C\sigma^{p_1}\varepsilon^{ap_1}\lambda).
\end{align}
Let us now use~\eqref{est:O-3} into the norm expression of the weighted Lorentz space to arrive
\begin{align}\nonumber
\|\Phi(\mathbf{M}_{\alpha}(|\nabla u|^p))\|_{L^{q,s}_{\omega}(\Omega)}^s & = q\int_0^{\infty} \lambda^{s-1} \left[\mathbf{D}^{\omega}_{\Phi(\mathbf{M}_{\alpha}(|\nabla u|^p))}(\lambda)\right]^{\frac{s}{q}} {d\lambda}  \\ \nonumber
& \le C \varepsilon^{\frac{s}{q}} q \int_0^{\infty} \lambda^{s-1} \left[ \mathbf{D}^{\omega}_{\Phi(\mathbf{M}_{\alpha}(|\nabla u|^p))}(C\varepsilon^{ap_1}\lambda)\right]^{\frac{s}{q}} {d\lambda} \\ \nonumber
& \qquad \qquad + C q\int_0^{\infty} \lambda^{s-1} \left[\mathbf{D}^{\omega}_{\Phi(\mathbf{M}_{\alpha}(|\mathbb{F}|^p))}(C\sigma^{p_1}\varepsilon^{ap_1}\lambda)\right]^{\frac{s}{q}} {d\lambda}.
\end{align}
By changing of variables, we may write
\begin{align}\nonumber
\|\Phi(\mathbf{M}_{\alpha}(|\nabla u|^p))\|_{L^{q,s}_{\omega}(\Omega)}^s & \le C \varepsilon^{\frac{s}{q}-sap_1} q \int_0^{\infty} \lambda^{s-1} \left[ \mathbf{D}^{\omega}_{\Phi(\mathbf{M}_{\alpha}(|\nabla u|^p))}(\lambda)\right]^{\frac{s}{q}} {d\lambda} \\ \nonumber
& \qquad \qquad + C \sigma^{-sp_1}\varepsilon^{-sap_1} q\int_0^{\infty} \lambda^{s-1} \left[\mathbf{D}^{\omega}_{\Phi(\mathbf{M}_{\alpha}(|\mathbb{F}|^p))}(\lambda)\right]^{\frac{s}{q}} {d\lambda} \\ \nonumber
& \le C \varepsilon^{\frac{s}{q}-sap_1} \|\Phi(\mathbf{M}_{\alpha}(|\nabla u|^p))\|_{L^{q,s}_{\omega}(\Omega)}^s  \\ \nonumber
& \qquad \qquad + C \sigma^{-sp_1}\varepsilon^{-sap_1} \|\Phi(\mathbf{M}_{\alpha}(|\mathbb{F}|^p))\|_{L^{q,s}_{\omega}(\Omega)}^s,
\end{align}
which implies to
\begin{align}\nonumber
\|\Phi(\mathbf{M}_{\alpha}(|\nabla u|^p))\|_{L^{q,s}_{\omega}(\Omega)} & \le C  \varepsilon^{\frac{1}{q}-ap_1} \|\Phi(\mathbf{M}_{\alpha}(|\nabla u|^p))\|_{L^{q,s}_{\omega}(\Omega)} \\ \label{est:L-O-4}
  & \hspace{2cm} +  C \sigma^{-p_1}\varepsilon^{-ap_1} \|\Phi(\mathbf{M}_{\alpha}(|\mathbb{F}|^p))\|_{L^{q,s}_{\omega}(\Omega)}. 
\end{align}
With the value of $a$ chosen as in~\eqref{choose-a}, one can fix $\varepsilon$ in~\eqref{est:L-O-4} small enough to observe that
\begin{align}\label{est:L-O-5}
\|\Phi(\mathbf{M}_{\alpha}(|\nabla u|^p))\|_{L^{q,s}_{\omega}(\Omega)} & \le   C_{*}   \|\Phi(\mathbf{M}_{\alpha}(|\mathbb{F}|^p))\|_{L^{q,s}_{\omega}(\Omega)}. 
\end{align}
By scaling ${\lambda}^{-1}{|\nabla u|^p}$, ${\lambda}^{-1}{|\mathbb{F}|^p}$ on weighted distribution inequality~\eqref{est:O-1} and using the convexity of $\Phi$, we obtain a similar estimate as in~\eqref{est:L-O-5} for any $\lambda>0$. More precisely, one gets that
\begin{align}\nonumber 
\left\|\Phi\left((C_{*}\lambda)^{-1}\mathbf{M}_{\alpha}(|\nabla u|^p)\right)\right\|_{L^{q,s}_{\omega}(\Omega)} & \le {C_{*}}^{-1} \left\|\Phi\left(\lambda^{-1}\mathbf{M}_{\alpha}(|\nabla u|^p)\right)\right\|_{L^{q,s}_{\omega}(\Omega)} \\ \nonumber
  & \le \left\|\Phi\left(\lambda^{-1}\mathbf{M}_{\alpha}(|\mathbb{F}|^p)\right)\right\|_{L^{q,s}_{\omega}(\Omega)}, \quad \forall \lambda >0.
\end{align}
This estimate yields that $H({\nabla u}) \subset \frac{1}{C_{*}} H({\mathbb{F}})$ which implies to~\eqref{eq:L-O}, where
\begin{align*}
H(f) = \left\{\lambda >0: \ \left\|\Phi\left({\lambda}^{-1}{\mathbf{M}_{\alpha}(|f|^p)}\right)\right\|_{L^{q,s}_{\omega}(\Omega)} \le 1\right\},
\end{align*}
with $f = \nabla u$ or $f = \mathbb{F}$. A slight changing of computation allows us to prove~\eqref{est:L-O-5} and~\eqref{eq:L-O} even in the case $s = \infty$. 
\end{proof}

\end{document}